\documentclass{amsart}
\usepackage{amssymb,euscript,amsmath, mathrsfs}
\usepackage[pdftex]{graphicx}
\usepackage[pdftex]{color}
\definecolor{navy}{rgb}{0,0,.5}
\definecolor{darkred}{rgb}{1,0,0}

\input xy
\xyoption{all}
\CompileMatrices

\def\bRR{{\mathbf R}}

\def\RR{{\mathbb R}}
\def\ZZ{{\mathbb Z}}

\def\fm{{\mathfrak m}}

\def\Spec{\operatorname{Spec}}
\def\Tor{\operatorname{Tor}}
\def\Hom{\operatorname{Hom}}
\def\ev{\operatorname{ev}}

\def\Trop{\operatorname{Trop}}
\def\trop{\operatorname{trop}}
\def\bTrop{\mathbf{Trop}}

\def\Star{\operatorname{Star}}

\def\Span{\operatorname{span}}
\def\Mink{\operatorname{Mink}}
\def\trdeg{\operatorname{trdeg}}

\def\an{\mathrm{an}}
\def\cD{{\mathcal D}}
\def\cI{{\mathcal I}}
\def\cL{{\mathcal L}}
\def\cM{{\mathcal M}}
\def\cO{{\mathcal O}}
\def\cP{{\mathcal P}}
\def\cQ{{\mathcal Q}}
\def\cS{{\mathcal S}}
\def\cU{{\mathcal U}}
\def\cT{{\mathcal T}}
\def\cW{{\mathcal W}}
\def\cX{{\mathcal X}}
\def\cY{{\mathcal Y}}
\def\cZ{{\mathcal Z}}

\def\<{{\langle}}
\def\>{{\rangle}}

\DeclareMathOperator{\conv}{conv}
\DeclareMathOperator{\codim}{codim}
\DeclareMathOperator{\length}{length}
\DeclareMathOperator{\rk}{rk}

\numberwithin{equation}{subsection}
\newtheorem{thm}[equation]{Theorem}
\newtheorem{prop}[equation]{Proposition}
\newtheorem{lem}[equation]{Lemma}
\newtheorem{cor}[equation]{Corollary}
\newtheorem*{bc}{Balancing Condition}
\newtheorem*{fdr}{Fan Displacement Rule}

\newtheorem*{thm*}{Theorem}
\newtheorem{thm2}{Theorem}[section]
\newtheorem{prop2}[thm2]{Proposition}
\theoremstyle{definition}
\newtheorem{defn}[equation]{Definition}

\newtheorem{ex}[equation]{Example}
\theoremstyle{remark}

\newtheorem{rem}[equation]{Remark}
\newtheorem{rem2}[thm2]{Remark}
\newtheorem{ex2}[thm2]{Example}

\begin{document}
\title{Lifting tropical intersections}
\author{Brian Osserman}
\address{Department of Mathematics, University of California, Davis, CA 95616}
\author{Sam Payne}
\address{Mathematics Department, Yale University, New Haven, CT 06511}

\begin{abstract}
We show that points in the intersection of the tropicalizations of subvarieties of a torus lift to algebraic intersection points with expected multiplicities, provided that the tropicalizations intersect in the expected dimension.  We also prove a similar result for intersections inside an ambient subvariety of the torus, when the tropicalizations meet inside a facet of multiplicity 1.  The proofs require not only the geometry of compactified tropicalizations of subvarieties of toric varieties, but also new results about the geometry of finite type schemes over non-noetherian valuation rings of rank 1.  In particular, we prove subadditivity of codimension and a principle of continuity for intersections in smooth schemes over such rings, generalizing well-known theorems over regular local rings.  An appendix on the topology of finite type morphisms may also be of independent interest.
\end{abstract}

\maketitle

\vspace{-20pt}

\tableofcontents

\section{Introduction}

Tropical geometry studies the valuations of solutions to polynomial equations, and may be thought of as a generalization of the theory of Newton polygons to multiple polynomials in multiple variables. It is natural to consider what conditions guarantee that if two closed subvarieties $X$ and $X'$ in a torus $T$ have points with the same valuation, then $X \cap X'$ contains a point of the given valuation. In the context of tropical geometry, this is the question of when tropicalization commutes with intersection.  The tropicalization of the intersection of two closed subvarieties of $T$ over a nonarchimedean field is contained in the intersection of their tropicalizations, but this containment is sometimes strict.  For instance, if $X$ is smooth in characteristic zero, and $X'$ is the translate of $X$ by a general torsion point $t$ in $T$, then $X$ and $X'$ are either disjoint or meet transversely, but their tropicalizations are equal.  The dimension of $X \cap X'$ is then strictly less than the dimension of $\Trop(X) \cap \Trop(X')$, so there are tropical intersection points that do not lift to algebraic intersection points.  In other cases, such as Example~\ref{ex:lines} below, the tropicalizations meet in a positive dimensional set that does not contain the tropicalization of any positive dimensional variety, so at most finitely many of the tropical intersection points can lift to algebraic intersection points.
Our main result, in its most basic form, says that the dimension of the tropical intersection is the only obstruction to such lifting.

Following standard terminology from intersection theory in algebraic geometry, we say that $\Trop(X)$ and $\Trop(X')$ meet \emph{properly} at a point $w$ if $\Trop(X) \cap \Trop(X')$ has codimension $\codim X + \codim X'$ in a neighborhood of $w$.  We say that $\Trop(X)$ meets $\Trop(X')$ properly if they meet properly at every point of their intersection, which may be empty.\footnote{Codimension is subadditive for intersections of tropicalizations, which means that $\Trop(X) \cap \Trop(X')$ has codimension at most $j + j'$ at every point, so $\Trop(X)$ meets $\Trop(X')$ properly if the intersection is either empty or has the smallest possible dimension.  See Proposition~\ref{prop:tropical subadditivity}.}

\begin{thm2} \label{basic}
Suppose $\Trop(X)$ meets $\Trop(X')$ properly at $w$.  Then $w$ is contained in the tropicalization of $X \cap X'$.
\end{thm2}

\noindent In other words, tropicalization commutes with intersections when the intersections have the expected dimension.  This generalizes a well-known result of Bogart, Jensen, Speyer, Sturmfels, and Thomas, who showed that tropicalization commutes with intersection when the tropicalizations meet transversely \cite[Lemma~3.2]{BJSST}.  Our proof, and the proofs of all of the other main results below, involves a reduction to the case where the base field is complete with respect to its nonarchimedean norm and the point $w$ is rational over the value group.  The reduction is based on the results of Appendix~\ref{app:finite type}, and should become a standard step in the rigorous application of tropical methods to algebraic geometry.  We emphasize that the theorems hold in full generality, over an arbitrary algebraically closed nonarchimedean field, and at any real point in the intersection of the tropicalizations.

Although Theorem \ref{basic} improves significantly on previously known results, it is still too restrictive for the most interesting potential applications. Frequently, $X$ and $X'$ are closed subvarieties of an ambient variety $Y$ inside the torus. In this case, one cannot hope that $\Trop(X)$ and $\Trop(X')$ will meet properly in the above sense. Instead, we say that the tropicalizations $\Trop(X)$ and $\Trop(X')$ meet properly at a point $w$ in $\Trop(Y)$ if the intersection
\[
\Trop(X) \cap \Trop(X') \subset \Trop(Y)
\]
has pure codimension $\codim_Y X + \codim_Y X'$ in a neighborhood of $w$. We extend Theorem \ref{basic} to proper intersections at suitable points of $\Trop(Y)$, as follows.

Recall that $\Trop(Y)$ is the underlying set of a polyhedral complex of pure dimension $\dim Y$, with a positive integer multiplicity assigned to each \textbf{facet}, by which we mean a maximal face.  We say that a point in $\Trop(Y)$ is \textbf{simple} if it is in the interior of a facet of multiplicity $1$.  

\begin{thm2} \label{main}
Suppose $\Trop(X)$ meets $\Trop(X')$ properly at a simple point $w$ in $\Trop(Y)$.  Then $w$ is contained in the tropicalization of $X \cap X'$.
\end{thm2}

\noindent Every point in the tropicalization of the torus is simple, so Theorem~\ref{basic} is the special case of Theorem~\ref{main} where $Y$ is the full torus $T$.   The hypothesis that $w$ be simple is necessary; see Section~\ref{sec:examples}.  We strengthen Theorem \ref{main} further by showing that where $\Trop(X)$ and $\Trop(X')$ meet properly, the facets of the tropical intersection appear with the expected multiplicities, suitably interpreted.  See Theorem~\ref{thm:main plus multiplicities} for a precise statement.

The proof of Theorem~\ref{main} is in two steps.  Roughly speaking, we lift first from tropical points to points in the initial degeneration, and then from the initial degeneration to the original variety.  Recall that the tropicalization of a closed subvariety $X$ of $T$ is the set of weight vectors $w$ such that the initial degeneration $X_w$ is nonempty in the torus torsor $T_w$ over the residue field.  

\begin{thm2}\label{thm:lift-to-degens} 
Suppose $\Trop(X)$ meets $\Trop(X')$ properly at a simple point $w$ in $\Trop(Y)$.  Then $X_w$ and $X'_w$ have nonempty proper intersection in the smooth variety $Y_w$.
\end{thm2}

\noindent If $w$ is a simple point of $\Trop(Y)$ then standard arguments show that $X_w$ meets $X'_w$ properly at every point in the intersection; the main content of the theorem is that $X_w \cap X'_w$ is nonempty.  The proof, given in Section~\ref{sec:trivial val}, uses extended tropicalizations and the intersection theory of toric varieties.  In Section~\ref{sec:dim}, we develop geometric techniques over valuation rings of rank 1 that permit lifting of such proper intersections at smooth points from the special fiber to the generic fiber.

\begin{thm2} \label{exploded main}
Suppose $X_w$ meets $X'_w$ properly at a smooth point $x$ of $Y_w$.  Then $x$ is contained in $(X \cap X')_w$.
\end{thm2}

\noindent In particular, if $X_w$ meets $X'_w$ properly at a smooth point of $Y_w$ then $w$ is contained in $\Trop(X \cap X')$.  If $w$ is defined over the value group of the nonarchimedean field then surjectivity of tropicalization says that $x$ can be lifted to a point of $X \cap X'$.  The dimension theory over valuation rings of rank 1 developed in Section~\ref{sec:dim} also gives a new proof of surjectivity of tropicalization, as well as density of tropical fibers.

Theorem~\ref{exploded main} is considerably stronger than Theorems~\ref{basic} and \ref{main}.  It often happens that initial degenerations meet properly even when tropicalizations do not.  The proof of this theorem follows standard arguments from dimension theory, but the dimension theory that is needed is not standard because the valuation rings we are working with are not noetherian.  We develop the necessary dimension theory systematically in Section~\ref{sec:dim}, using noetherian approximation to prove a version of the Krull principal ideal theorem and then deducing subadditivity of codimension for intersections in smooth schemes of finite type over valuation rings of rank 1.  Furthermore, we prove a principle of continuity, showing that intersection numbers are well-behaved in families over valuation rings of rank 1.  In Section~\ref{sec:main}, we apply this principle of continuity to prove a stronger version of Theorem~\ref{main} with multiplicities and then extend all of these results to proper intersections of three or more subvarieties of $T$.  One special case of these lifting results with multiplicities, for complete intersections, has been applied by Rabinoff in an arithmetic setting to construct canonical subgroups for abelian varieties over $p$-adic fields \cite{Rabinoff12b}.

\begin{rem2}
In addition to their intrinsic appeal, our results are motivated by the possibility of applications to proving correspondence theorems such as the one proved by Mikhalkin for plane curves \cite{Mikhalkin05}.  Mikhalkin considers curves of fixed degree and genus subject to constraints of passing through specified points.  He shows, roughly speaking, that a tropical plane curve that moves in a family of the expected dimension passing through specified points lifts to a predictable number of algebraic curves passing through prescribed algebraic points.  Here, we consider points subject to the constraints of lying inside closed subvarieties $X$ and $X'$ and prove the analogous correspondence---if a tropical point moves in a family of the expected dimension inside $\Trop(X) \cap \Trop(X')$ then it lifts to a predictable number of points in $X \cap X'$.  One hopes that Mikhalkin's correspondence theorem will eventually be reproved and generalized using Theorem~\ref{main} on suitable tropicalizations of moduli spaces of curves.  Similarly, one hopes that Schubert problems can be answered tropically using Theorem~\ref{main} on suitable tropicalizations of Grassmannians and flag varieties. The potential for such applications underlines the importance of having the flexibility to work with intersections inside ambient subvarieties of the torus.
\end{rem2}

\smallskip

\noindent \textbf{Acknowledgments.}  
We thank B.~Conrad, D.~Eisenbud, W.~Heinzer, D.~Rydh, and B.~Ulrich for helpful conversations, and are most grateful to J.~Rau and M.~Rojas for insightful comments on an earlier draft of this work. We also thank the referee for a careful reading.

\bigskip

\noindent The first author was partially supported by a fellowship from the National Science Foundation.  The second author was partially supported by the Clay Mathematics Institute and NSF DMS grant 1068689.  Part of this research was carried out during the special semesters on Algebraic Geometry and Tropical Geometry at MSRI, in Spring and Fall 2009.

\section{Preliminaries}\label{sec:prelim}

Let $K$ be an algebraically closed field, and let $\nu:K^* \to \RR$ be a valuation with value group $G$.  Let $R$ be the valuation ring, with maximal ideal $\fm$, and residue field $k = R/\fm$.  Since $K$ is algebraically closed, the residue field $k$ is algebraically closed and the valuation group $G$ is divisible.  In particular, $G$ is dense in $\RR$ unless the valuation $\nu$ is trivial, in which case $G$ is zero.  Typical examples of such nonarchimedean fields in equal characteristic are given by the generalized power series fields $K=k((t^G))$, whose elements are formal power series with coefficients in the algebraically closed field $k$ and exponents in $G$, where the exponents occurring in any given series are required to be well-ordered \cite{Poonen93}.

Let $\cT$ be an algebraic torus of dimension $n$ over $R$, with character lattice $M \cong \ZZ^n$, and let $N = \Hom(M,\ZZ)$ be the dual lattice.  We write $N_G$ for $N \otimes G$, and $N_{\RR}$ for $N \otimes \RR$, the real vector space of linear functions on the character lattice.  We treat $M$ and $N$ additively and write $x^u$ for the character associated to a lattice point $u$ in $M$, which is a monomial in the Laurent polynomial ring $K[M]$.  We write $T$ for the associated torus over $K$.

In this section, we briefly review the basic properties of initial degenerations and tropicalizations, as well as the relationship between tropical intersections and intersections in suitable toric compactifications of $T$.

\subsection{Initial degenerations} Each vector $w$ in $N_\RR$ determines a weight function on monomials, where the $w$-weight of $a x^u$ is $\nu(a) + \<u, w\>$.  The tilted group ring $R[M]^w \subset K[M]$ consists of the Laurent polynomials $\sum a_u x^u$ in which every monomial has nonnegative $w$-weight.  We then define $\cT^w = \Spec R[M]^w$.  If $\nu$ is nontrivial, then $\cT^w$ is an integral model of $T$, which means that it is a scheme over $R$ whose generic fiber $\cT^w \times_{\Spec R} \Spec K$ is naturally identified with $T$.  In general, the scheme $\cT^w$ carries a natural $\cT$-action, and is a $\cT$-torsor if $w$ is in $N_G$.

\begin{rem}
If $\nu$ is nontrivial and $w$ is not in $N_G$, or if $\nu$ is trivial and $w$ does not span a rational ray, then $\cT^w$ is not of finite type over $\Spec R$.  Some additional care is required in handling these schemes, but no major difficulties arise for the purposes of this paper.  The special fiber $\cT^w_k$ is still finite type over $k$, and is a torsor over a quotient torus of $T_k$.  The basic properties of the schemes $\cT^w$ and their closed subschemes may be understood by passing to a valued extension field with value group $G'$ such that $w$ is in $N_{G'}$ and analyzing how these schemes and their special fibers transform under such extensions. This analysis is carried out in Appendix~\ref{app:finite type}.  See, in particular, Theorems~\ref{thm:integral base change} and \ref{thm:initial base change}, and Remark~\ref{rmk:properties of base change}.
\end{rem}

For arbitrary $w$ in $N_\RR$, the scheme $\cT^w$ is reduced and irreducible, flat over $\Spec R$, and contains $T$ as a dense open subscheme.  We define $T_w$ to be the closed subscheme cut out by monomials of strictly positive $w$-weight.  If the valuation is nontrivial, then $T_w$ is the special fiber of $\cT^w$.

Let $X$ be a closed subscheme of $T$, of pure dimension $d$.  Let $\cX^w$ denote the closure of $X$ in $\cT^w$.

\begin{defn} 
The \textbf{initial degeneration} $X_w$ is the closed subscheme of $T_w$ obtained by intersecting with $\cX^w$.
\end{defn}

\noindent The terminology reflects the fact that $X_w$ is cut out by residues of initial terms (lowest-weight monomials) of Laurent polynomials in the ideal of $X$.  If the valuation is nontrivial, then $\cX^w$ is an integral model of $X$, and $X_w$ is the special fiber of this model.

One consequence of Gr\"obner theory is that the space of weight vectors $N_\RR$ can be decomposed into finitely many polyhedral cells so that the initial degenerations are essentially invariant at points of $N_G$ on the relative interior of each cell, as discussed in more detail in the following subsection. Roughly speaking, the cells are cut out by inequalities whose linear terms have integer coefficients, and whose constant terms are in the value group $G$.  More precisely, the cells are integral $G$-affine polyhedra, defined as follows.

\begin{defn} 
An \textbf{integral $G$-affine polyhedron} in $N_\RR$ is the solution set of a finite number of inequalities 
\[
\<u,v\> \leq b,
\]
 with $u$ in the lattice $M$ and $b$ in the value group $G$.  
\end{defn}
 
\noindent An integral $G$-affine polyhedral complex $\Sigma$ is a polyhedral complex consisting entirely of integral $G$-affine polyhedra.  In other words, it is a finite collection of integral $G$-affine polyhedra such that every face of a polyhedron in $\Sigma$ is itself in $\Sigma$, and the intersection of any two polyhedra in $\Sigma$ is a face of each.  Note that if $G$ is zero, then an integral $G$-affine polyhedron is a rational polyhedral cone, and an integral $G$-affine polyhedral complex is a fan.

\subsection{Tropicalization} \label{sec:tropicalization}
Following Sturmfels, we define the \textbf{tropicalization} of $X$ to be
\[
\Trop(X) = \{ w \in N_\RR \ | \ X_w \mbox{ is nonempty}\}.
\]
The foundational theorems of tropical geometry, due to the work of many authors, are the following.\footnote{See \cite{BieriGroves84}, \cite{SpeyerSturmfels04}, and Propositions~2.1.4, Theorem~2.2.1, and Proposition~2.4.5 of \cite{SpeyerThesis} for proofs of (1) and (2).  The first proposed proof of (3), due to Speyer and Sturmfels, contained an essential gap that is filled by Theorem~\ref{thm:closed-pts}, below.  Other proofs of (3) have appeared in \cite[Theorem~4.2]{Draisma08}, \cite[Proposition~4.14]{Gubler12} and \cite{tropicalfibers, TropicalFibers-Correction}. See also Remark~\ref{rem:second gap}.}

\begin{enumerate}
\item The tropicalization $\Trop(X)$ is the underlying set of an integral $G$-affine polyhedral complex of pure dimension $d$.
\item The integral $G$-affine polyhedral structure on $\Trop(X)$ can be chosen so that the initial degenerations $X_w$ and $X_{w'}$ are $\cT_k$-affinely equivalent for any $w$ and $w'$ in $N_G$ in the relative interior of the same face.
\item The image of $X(K)$ under the natural tropicalization map $\trop: T(K) \rightarrow N_\RR$ is exactly $\Trop(X) \cap N_G$.  
\end{enumerate}

\noindent Here, two subschemes of the $\cT_k$-torsors $T_w$ and $T_{w'}$ are said to be $\cT_k$-affinely equivalent if they are identified under some $\cT_k$-equivariant choice of isomorphism $T_w \cong T_{w'}$.  In fact, a stronger statement holds: for any points $w$ and $w'$ in $N_G$, there exists a $\cT_k$-equivariant isomorphism $T_w \cong T_{w'}$ which sends $X_w$ to $X_{w'}$ for all $X$ such that $w$ and $w'$ are in the relative interior of the same face of $\Trop(X)$.  If the valuation is nontrivial, then it follows from (3) that $\Trop(X)$ is the closure of the image of $X(K)$ under the tropicalization map.

For any extension of valued fields $L|K$, the tropicalization of the base change $\Trop(X_L)$ is exactly equal to $\Trop(X)$ \cite[Proposition~6.1]{analytification}.  In particular, for any algebraically closed extension $L|K$ with nontrivial valuation, $\Trop(X)$ is the closure of the image of $X(L)$.  Furthermore, if we extend to some $L$ that is complete with respect to its valuation then the tropicalization map on $X(L)$ extends naturally to a continuous map on the nonarchimedean analytification of $X_L$, in the sense of Berkovich \cite{Berkovich90}, whose image is exactly $\Trop(X)$ \cite{analytification}.  It follows, by reducing to the case of a complete field, that $\Trop(X)$ is connected if $X$ is connected, since the analytification of a connected scheme over a complete nonarchimedean field is connected.

\begin{rem}
The natural tropicalization map from $T(K)$ to $N_\RR$ takes a point $t$ to the linear function $u \mapsto \nu \circ \ev_t x^u$, and can be understood as a coordinatewise valuation map, as follows.  The choice of a basis for $M$ induces isomorphisms  $T\cong (K^*)^n$ and $N_{\RR} \cong \RR^n$.  In such coordinates, the tropicalization map sends $(t_1,\dots, t_n)$ to  $(\nu(t_1),\dots,\nu(t_n))$.
\end{rem}

There is no canonical choice of polyhedral structure on $\Trop(X)$ satisfying (2) in general, but any refinement of such a complex again satisfies (2).  Throughout the paper, we assume that such a polyhedral complex with underlying set $\Trop(X)$ has been chosen.  We refer to its faces and facets as faces and facets of $\Trop(X)$, and refine the complex as necessary.

The tropicalization of a closed subscheme of a torus torsor over $K$ is well-defined up to translation by $N_G$.  In particular, if the value group is zero, then the tropicalization is well-defined, and is the underlying set of a rational polyhedral fan.  An important special case is the tropicalization of an initial degeneration.  The valuation $\nu$ induces the trivial valuation on the residue field $k$, and there is a natural identification of $\Trop(X_w)$ with the star of $w$ in $\Trop(X)$ \cite[Proposition~2.2.3]{SpeyerThesis}.  This star is, roughly speaking, the fan one sees looking out from $w$ in $\Trop(X)$; it is constructed by translating $\Trop(X)$ so that $w$ is at the origin and taking the cones spanned by faces of $\Trop(X)$ that contain $w$.

\subsection{Tropical multiplicities} \label{sec:tropical mults}
If $w$ and $w'$ are points of $N_G$ in the relative interior of the same cone $\sigma$, then $X_w$ and $X'_w$ are $\cT_k$-affinely equivalent.  In particular, they are isomorphic as schemes, so the sum of the multiplicities of their irreducible components are equal.  By Theorem~\ref{thm:initial base change}, the sum of the multiplicities of the irreducible components of the initial degeneration at $w$ is invariant under extensions of valued fields, for any $w$ in $N_\RR$.  Since any point in $N_\RR$ becomes rational over the value group after a suitable extension, it follows that this sum is independent of the choice of $w$ in the relative interior of $\sigma$.  We will be concerned with this sum only in the case where $\sigma$ is a facet.  For applications of tropical multiplicities at points that are not in the relative interior of a facet, see \cite{BPR11}.

\begin{defn}
The \textbf{tropical multiplicity} $m(\sigma)$ of a facet $\sigma$ in $\Trop(X)$ is the sum of the multiplicities of the irreducible components of $X_w$, for $w$ in the relative interior of $\sigma$.
\end{defn}

For $w \in N_G$, the multiplicities on the facets of $\Trop(X_w)$ agree with those on the facets of $\Trop(X)$ that contain $w$.  This is because the initial degeneration of $X_w$ at a point in the relative interior of a given facet is $\cT_k$-affinely equivalent to the initial degeneration of $X$ at a point in the relative interior of the corresponding facet of $\Trop(X)$.  See \cite[Proposition~2.2.3]{SpeyerThesis} and \cite[Proposition~10.9]{Gubler12}.

Points in the relative interiors of facets of multiplicity 1 will be particularly important for our purposes.

\begin{defn} A \textbf{simple point} in $\Trop(X)$ is a point in the relative interior of a facet of multiplicity 1.
\end{defn}

\noindent If $w \in N_G$ is a simple point, then the initial degeneration $X_w$ is isomorphic to a $d$-dimensional torus \cite[Proposition~2.2.4]{SpeyerThesis}, and initial degenerations at arbitrary simple points are isomorphic to tori of dimension at most $d$.  The latter follows from the case where $w$ is in $N_G$ by choosing a suitable extension of valued fields and applying Theorem~\ref{thm:initial base change}.  In particular, the initial degeneration at an arbitrary simple point is smooth.  

The importance of simple points, roughly speaking, is that intersections of tropicalizations of closed subvarieties of $X$ at simple points behave just like intersections in $N_\RR$ of tropicalizations of closed subvarieties of $T$.  See, for instance, the proof of Theorem~\ref{main} in Section~\ref{sec:main}.

\subsection{Minkowski weights and the fan displacement rule} \label{sec:Mink}

We briefly review the language of Minkowski weights and their ring structure, in which the product is given by the fan displacement rule.  These tools are the basis of both toric intersection theory, as developed in \cite{FultonSturmfels97}, and the theory of stable tropical intersections \cite{RST, Mikhalkin06}, discussed in the following sections.  We refer the reader to these original papers for further details.  Although the theory of Minkowski weights extends to arbitrary complete fans, we restrict to unimodular fans, corresponding to smooth toric varieties, which suffice for our purposes.  Fix a complete unimodular fan $\Sigma$ in $N_{\RR}$.

\begin{defn} A {\it Minkowski weight} of codimension $j$ on $\Sigma$ is a function $c$ that assigns an integer $c(\sigma)$ to each codimension $j$ cone $\sigma$ in $\Sigma$, and satisfies the following balancing condition at every codimension $j+1$ cone.
\end{defn}

\begin{bc}
Let $\tau$ be a cone of codimension $j+1$ in $\Sigma$.  Say $\sigma_1, \ldots, \sigma_r$ are the codimension $j$ cones of $\Sigma$ that contain $\tau$, and let $v_i$ be the primitive generator of the image of $\sigma_i$ in $N_\RR/ \Span(\tau)$.  Then $c$ is balanced at $\tau$ if
\[
c(\sigma_1) v_1  + \cdots + c(\sigma_r) v_r = 0
\]
in $N_\RR / \Span(\tau)$.
\end{bc}

\noindent Tropicalizations are one interesting source of Minkowski weights.  If $X$ is a closed subscheme of pure codimension $j$ in $T$ and $\Trop(X)$ is a union of cones in $\Sigma$, then the function $c$ given by
\[
c(\sigma) = \left \{ \begin{array}{ll} m(\sigma) & \mbox{ if } \sigma \subset \Trop(X). \\
							0			 &  \mbox{ otherwise.}
							\end{array} \right.
\]
satisfies the balancing condition, and hence is a Minkowski weight of codimension $j$.  This correspondence between tropicalizations and Minkowski weights is compatible with intersection theory in the toric variety $Y(\Sigma)$.  See Section~\ref{sec:compatibility}, below.

We write $\Mink^j(\Sigma)$ for the group of Minkowski weights of codimension $j$ on $\Sigma$.  The direct sum
\[
\Mink^*(\Sigma) = \bigoplus_j \Mink^j(\Sigma)
\]
naturally forms a graded ring, whose product is given by the fan displacement rule, as follows.  Let $v$ be a vector in $N_\RR$ that is sufficiently general so that for any two cones $\sigma$ and $\sigma'$ in $\Sigma$, the displaced cone $\sigma' + v$ meets $\sigma$ properly.  The set of such vectors is open and dense in $N_\RR$; it contains the complement of a finite union of linear spaces.  For each cone $\sigma$ in $\Sigma$, let $N_\sigma$ denote the sublattice of $N$ generated by $\sigma \cap N$.  Note that if $\sigma$ 
intersects the displaced cone $\sigma'+v$ then $\sigma$ and $\sigma'$ together span $N_\RR$, so the index $[N : N_\sigma + N_{\sigma'}]$ is finite.

\begin{fdr}
The product $c \cdot c'$ is the Minkowski weight of co\-di\-mension $j + j'$ given by
\[
(c \cdot c')(\tau) = \sum_{\sigma, \sigma'} \ [N : N_\sigma + N_{\sigma'}] \cdot c(\sigma) \cdot c'(\sigma'),
\]
where the sum is over cones $\sigma$ and $\sigma'$ containing $\tau$, of codimensions $j$ and $j'$ respectively, such that the intersection $\sigma \cap (\sigma' + v)$ is nonempty.
\end{fdr}

\noindent The balancing conditions on $c$ and $c'$ ensure that the product is independent of the choice of displacement vector $v$, and also satisfies the balancing condition.

\subsection{Toric intersection theory}\label{sec:toric-intersect}

We briefly recall the basics of intersection theory in smooth complete toric varieties.  As in the previous section, we fix a complete unimodular fan $\Sigma$ in $N_\RR$.  Let $Y(\Sigma)$ be the associated smooth complete toric variety.  A codimension $j$ cycle in $Y(\Sigma)$ is a formal sum of codimension $j$ closed subvarieties, with integer coefficients, and we write $A^j(Y(\Sigma))$ for the Chow group of codimension $j$ cycles modulo rational equivalence.  Let $X$ be a closed subscheme of pure codimension $j$ in $Y(\Sigma)$.  We write $[X]$ for the class it represents in $A^j(Y(\Sigma))$; by definition, this is the sum over the irreducible components $Z$ of $X$ of the length of $X$ along $Z$ times the class $[Z]$.

Intersection theory gives a natural ring structure on the direct sum
\[
A^*(Y(\Sigma)) = \bigoplus_j A^j(Y(\Sigma)).
\]
If $X$ and $X'$ have complementary dimension, which means that $\dim X + \dim X' = \dim Y(\Sigma)$, then the product $[X] \cdot [X']$ is a zero-dimensional cycle class.  The \textbf{degree} of this class, which is the sum of the coefficients of any representative cycle, is denoted $\deg([X] \cdot [X'])$.  In particular, if $\sigma$ is a cone of codimension $j$ in $\Sigma$, with $V(\sigma)$ the associated closed $T$-invariant subvariety, then $X$ and $V(\sigma)$ have complementary dimension.  We write $c_X$ for the induced function on codimension $j$ cones, given by
\[
c_X(\sigma) = \deg([X] \cdot [V(\sigma)]).
\]
The fact that the intersection product respects rational equivalence ensures that $c_X$ satisfies the balancing condition and is therefore a Minkowski weight of codimension $j$ on $\Sigma$.  The main result of \cite{FultonSturmfels97} then says that there is a natural isomorphism of rings
\[
A^*(Y(\Sigma)) \xrightarrow{\sim} \Mink^*(\Sigma)
\]
taking the Chow class $[X]$ to the Minkowski weight $c_X$.  

\smallskip

The theory of refined intersections says that the product $[X] \cdot [X']$ of the classes of two pure-dimensional closed subschemes of a smooth variety $Y$ is not only a well-defined cycle class in that smooth variety, but also the Gysin push-forward of a well-defined cycle class in the intersection $X \cap X'$.  If $X$ and $X'$ meet properly, then this class has codimension zero, and hence can be written uniquely as a formal sum of the components of $X \cap X'$.  The coefficient of a component $Z$ of the intersection is called the \textbf{intersection multiplicity} of $X$ and $X'$ along $Z$, and is denoted $i(Z, X \cdot X'; Y)$.  It is a positive integer, less than or equal to the length of the scheme theoretic intersection of $X$ and $X'$ along the generic point of $Z$, by \cite[Proposition~7.1]{IT}, and can be computed by Serre's alternating sum formula
\[
i(Z,X \cdot X';Y)=\sum_j (-1)^j \length_{\cO_{Y,Z}} \Tor^j_{\cO_{Y,Z}}(\cO_X, \cO_{X'}).
\]
See \cite[Chapter~8]{IT} for further details on the theory of refined intersections in smooth varieties.

\begin{defn}
If $X$ and $X'$ meet properly in a smooth variety $Y$, then the \textbf{refined intersection cycle} is
\[
X \cdot X' = \sum_Z i(Z, X \cdot X'; Y) \, Z,
\]
where the sum is over all irreducible components $Z$ of $X \cap X'$.
\end{defn}

\noindent We will consider such refined intersections inside the torus $T$, as well as in its smooth toric compactifications. Finally, it is useful to work not only with tropicalizations of closed subschemes of $T$, but also with tropicalizations of these refined intersection cycles.

\begin{defn} \label{def:cycles}
Let $a_1 Z_1 + \cdots + a_r Z_r$ be a pure-dimensional cycle in $T$, with positive integer coefficients $a_i$. The \textbf{tropicalization} of this cycle is the union
\[
\Trop(a_1 Z_1 + \cdots + a_r Z_r) =  \Trop(Z_1) \cup \cdots \cup \Trop(Z_r)
\]
with multiplicities
\[
m(\tau) = a_1 m_{Z_1}(\tau) + \cdots + a_r m_{Z_r}(\tau),
\]
where $m_{Z_i}(\tau)$ is defined to be zero if $\Trop(Z_i)$ does not contain $\tau$.
\end{defn}

\subsection{Stable tropical intersections}  \label{sec:stable}

Many papers have introduced and studied analogues of intersection theory in tropical geometry, including \cite{AllermannRau10, BrugalleLopez11, Katz12, Mikhalkin06, RST, Tabera08}.  Here we are interested in only one basic common feature of these theories, the stable tropical intersection, and its compatibility with intersection theory in toric varieties.  The key point is that stable tropical intersections are defined combinatorially, depending only on the polyhedral geometry of the tropicalizations, by a local displacement rule.  Compatibility with the multiplication rule for Minkowski weights and algebraic intersections of generic translates is discussed in the following section.

Given closed subschemes $X$ and $X'$ of $T$, of pure codimensions $j$ and $j'$, respectively, the stable intersection $\Trop(X) \cdot \Trop(X')$ is a polyhedral complex of pure codimension $j + j'$, with support contained in the set-theoretic intersection $\Trop(X) \cap \Trop(X')$, and with appropriate multiplicities on its facets, satisfying the balancing condition.  Many different closed subschemes of $T$ may have the same tropicalization, but the stable intersection depends only on the tropicalizations. 

Suppose the valuation is trivial, so $\Trop(X)$ and $\Trop(X')$ are fans, and let $\tau$ be a face of $\Trop(X) \cap \Trop(X')$ of codimension $j + j'$ in $N_\RR$.  The \textbf{tropical intersection multiplicity} of $\Trop(X)$ and $\Trop(X')$ along $\tau$, is given by the fan displacement rule
\[
i(\tau, \Trop(X) \cdot \Trop(X')) = \sum_{\sigma, \sigma'} [N: N_\sigma + N_{\sigma'}] \cdot m(\sigma) \cdot m(\sigma'),
\]
where the sum is over facets $\sigma$ and $\sigma'$ of $\Trop(X)$ and $\Trop(X')$, respectively, such that the intersection $\sigma \cap (\sigma' + v)$ is nonempty, and $v \in N$ is a fixed generic displacement vector, as in Section~\ref{sec:Mink}.

\smallskip

In the general case, where the valuation may be nontrivial, the tropical intersection multiplicities are defined similarly, by a local displacement rule, as follows.  Let $\tau$ be a face of codimension $j + j'$ in $\Trop(X) \cap \Trop(X')$, and for each face $\sigma$ containing $\tau$, let $N_\sigma$ be the sublattice of $N$ parallel to the affine span of $\sigma$.

\begin{defn}
The tropical intersection multiplicity $i(\tau, \Trop(X) \cdot \Trop(X'))$ is given by the \textbf{local displacement rule}
\[
i(\tau, \Trop(X) \cdot \Trop(X')) = \sum_{\sigma, \sigma'} [N: N_\sigma + N_{\sigma'}] \cdot m(\sigma) \cdot m(\sigma'),
\]
where $v \in N$ is a generic displacement vector, as above, and the sum is over facets $\sigma$ and $\sigma'$ of $\Trop(X)$ and $\Trop(X')$, respectively, such that the intersection $\sigma \cap (\sigma' + \epsilon v)$ is nonempty for $\epsilon$ sufficiently small and positive.
\end{defn}

\begin{defn}
 The \textbf{stable tropical intersection}, denoted $\Trop(X) \cdot \Trop(X')$, is the union of those faces $\tau$ such that $i(\tau, \Trop(X) \cdot \Trop(X'))$ is positive, weighted by their tropical intersection multiplicities.
\end{defn}

\subsection{Compatibility of toric and stable tropical intersections}  \label{sec:compatibility}

As discussed in Section~\ref{sec:Mink}, tropical multiplicities satisfy the balancing condition, and hence give Minkowski weights.  The stable tropical intersections are compatible with multiplication of Minkowski weights, and algebraic intersections of generic translates in the torus, as follows.

Let $X$ be a closed subscheme of pure codimension $j$ in $T$.  Suppose the valuation is trivial, so $\Trop(X)$ is a fan.  After subdividing, we may assume this fan is unimodular, and extends to a complete unimodular fan $\Sigma$ in $N_\RR$.  Let $Y(\Sigma)$ be the associated smooth complete toric variety, and let $\overline X$ be the closure of $X$ in $Y(\Sigma)$.  Because $\Sigma$ contains $\Trop(X)$ as a subfan, the closure $\overline X$ meets every torus orbit $O_\tau$ properly, and the intersection is nonempty if and only if $\tau$ is contained in $\Trop(X)$.  This follows from the basic properties of extended tropicalizations, given in \cite[Proposition~3.7]{analytification}.

\begin{rem} \label{rem:meeting-orbits}
Indeed the closure $\overline X$ is proper and meets every torus orbit properly if and only if $\Sigma$ contains a subfan whose support is $\Trop(X)$.  See \cite[Section~14]{Gubler12}.  
 \end{rem}
 
 Furthermore, for any facet $\sigma$ of $\Trop(X)$, the tropical multiplicity $m(\sigma)$ is equal to the intersection number of the closure $\overline X$ with the torus invariant subvariety corresponding to $\sigma$, that is
\[
m(\sigma) = \deg([\overline X] \cdot V(\sigma)).
\]
In other words, 
\[
c_{\overline X}(\sigma) = \left \{ \begin{array}{ll} m(\sigma) & \mbox{ if } \sigma \subset \Trop(X). \\
											0  & \mbox{ otherwise.}
											\end{array} \right.
\]
See \cite[Lemma~3.2]{SturmfelsTevelev08} for the case of a tropical compactification and \cite[Lemma~2.3]{KatzPayne11} for the general case.

\begin{lem}\label{lem:linearity}
Suppose the valuation is trivial, and let $X$ be a pure-dimensional closed subscheme of $T$.  Then the tropicalization $\Trop(X)$ is equal to the tropicalization of the fundamental cycle $\Trop([X])$.
\end{lem}

Recall that the multiplicities of facets in the tropicalization of a cycle are defined by linearity; see Definition~\ref{def:cycles}.

\begin{proof}
The tropicalizations agree set-theoretically, because both are the closures of the images of $X(L)$ where $L|K$ is an algebraically closed valued extension field with nontrivial valuation. The multiplicities agree because the tropical multiplicity of a facet $\sigma$ in $\Trop(X)$ is $m(\sigma) = \deg([\overline X] \cdot V(\sigma))$, and this degree is a linear function of the cycle $[\overline X]$.
\end{proof}

\begin{rem}
The tropicalization of a pure-dimensional closed subscheme of $T$ also agrees with the tropicalization of its underlying cycle in the case of a nontrivial valuation, but this requires considerably more work.  This is deduced from our general results on intersection multiplicities over valuation rings of rank 1 in Corollary~\ref{cor:cycles}. See also \cite[Section~13]{Gubler12} for an analytic proof using nonarchimedean GAGA and Gubler's theory of cycles and Cartier divisors on affinoid spaces.
\end{rem}

Now consider the general case, where the valuation may be nontrivial, and let $w \in N_G$ be a point of $\Trop(X)$.  After subdividing, we may assume that the star of $w$ in $\Trop(X)$ is unimodular, and extends to a complete unimodular fan $\Sigma$.  We write $\overline X_w$ for the closure of $X_w$ in $Y(\Sigma)$.  

\begin{defn}
If $\sigma$ is a facet of $\Trop(X)$ that contains $w$ then let 
\[
\sigma_w = \RR_{\geq 0} (\sigma - w)
\]
be the corresponding cone in $\Sigma$.  
\end{defn}

Recall that the star of $w$ in $\Trop(X)$ is the tropicalization of the initial degeneration $X_w$ and, as discussed in Section~\ref{sec:tropical mults}, the initial degeneration of $X_w$ at a point in the relative interior of $\sigma_w$ is $\cT_k$-affinely equivalent to the initial degeneration of $X$ at a point of $N_G$ in the relative interior of $\sigma$.  In particular, the tropical multiplicity $m(\sigma_w)$ in $\Trop(X_w)$ is equal to the tropical multiplicity $m(\sigma)$.

\begin{lem}
Let $\gamma$ be a face of codimension $j$ in $\Sigma$.  Then
\[
c_{\overline X_w}(\gamma) = \left \{ \begin{array}{ll} m(\sigma) & \mbox{ if } \gamma = \sigma_w, \mbox{ for some } \sigma \subset \Trop(X). \\
											0  & \mbox{ otherwise.}
											\end{array} \right.
											\]
\end{lem}

\begin{proof}
This follows from the case of the trivial valuation and the fact that the tropical multiplicity $m(\sigma)$ in $\Trop(X)$ is equal to the tropical multiplicity $m(\sigma_w)$ in $\Trop(X_w)$.
\end{proof}

This correspondence between tropical multiplicities and Minkowski weights is compatible with intersections; stable tropical intersections corresponding to products of Minkowski weights, as follows.  Suppose $X'$ is a closed subscheme of pure codimension $j'$ in $T$, and $\tau$ is a face of codimension $j + j'$ in $\Trop(X) \cap \Trop(X')$ that contains $w$.  We may assume that the star of $w$ in $\Trop(X')$ is a subfan of the unimodular fan $\Sigma$, and we write $\overline X'_w$ for the closure of $X'_w$ in $Y(\Sigma)$.

\begin{prop} \label{prop:local int}
Let $\gamma$ be a cone of codimension $j + j'$ in $\Sigma$.  Then
\[
(c_{\overline X_w} \cdot c_{\overline X'_w})(\gamma) = \left \{ \begin{array}{ll} i(\tau, \Trop(X) \cdot \Trop(X')) & \mbox{ if } \gamma = \tau_w, \\
											0  & \mbox{ otherwise,}
											\end{array} \right.
											\]
\end{prop}

\begin{proof}
Suppose $\gamma = \tau_w$.  We claim that the displacement rules giving $(c_{\overline X_w} \cdot c_{\overline X'_w})(\gamma)$ and $i(\tau, \Trop(X) \cdot \Trop(X'))$ agree, term by term.  If $\sigma$ and $\sigma'$ are facets of $\Trop(X)$ and $\Trop(X')$, respectively, that contain $\tau$, then $\sigma$ meets $\sigma' + \epsilon v$ for small positive $\epsilon$ if and only if $\sigma_w$ meets $\sigma'_w + v$.  Then we have an equality of summands in the displacement rules,
\[
[N: N_\sigma + N_{\sigma'}] \cdot m(\sigma) \cdot m(\sigma') = [N: N_{\sigma_w} + N_{\sigma'_w}] \cdot c_{\overline X_w}(\sigma_w) \cdot c_{\overline X'_w}(\sigma'_w),
\]
because $N_\sigma + N_\sigma'$, $m(\sigma)$, and $m(\sigma')$ are equal to $ N_{\sigma_w} + N_{\sigma'_w}$, $c_{\overline X_w}(\sigma_w)$, and $c_{\overline X'_w}(\sigma'_w)$, respectively.  The Minkowski weights $c_{\overline X_w}$ and $c_{\overline X'_w}$ vanish on cones that do not come from facets of $\Trop(X)$ and $\Trop(X')$, respectively.  Therefore, all nonzero summands in the displacement rules are accounted for in the equality above, and the proposition follows.
\end{proof}

\noindent It follows that the star of $w$ in the stable tropical intersection $\Trop(X) \cdot \Trop(X')$ is exactly the union of the faces of $\Sigma$  on which $(c_{\overline X_w} \cdot c_{\overline X'_w})$ is positive, with multiplicities given by this product of Minkowski weights.

We now return to the case of the trivial valuation.  As above, we choose $\Sigma$ to be a complete unimodular fan in $N_\RR$ that contains $\Trop(X)$ and $\Trop(X')$ as subfans.

\begin{prop} \label{prop:cycles}
Suppose the valuation is trivial.  Assume $X$ meets $X'$ properly in $T$ and, moreover,  $\overline X \cap V(\sigma)$ meets $\overline X' \cap V(\sigma)$ properly in $V(\sigma)$ for each $\sigma \in \Sigma$.  

Then, for any face $\tau$ of $\Trop(X) \cap \Trop(X')$ of codimension $j + j'$ in $N_\RR$, the tropical intersection multiplicity along $\tau$ is equal to the weighted sum of multiplicities of tropicalizations
\[
i(\tau,\Trop(X)\cdot \Trop(X'))= 
\sum_Z i(Z, X \cdot X'; T) m_{Z}(\tau),
\]
where the sum is over components $Z$ of $X \cap X'$ such that $\Trop(Z)$ contains $\tau$, with tropical multiplicity $m_Z(\tau)$.
\end{prop}

\begin{proof} 
First of all, by the preceding proposition,
\[
i(\tau, \Trop(X) \cdot \Trop(X')) = (c_{\overline X} \cdot c_{\overline X'})(\tau).
\]
Next, the isomorphism from $\Mink^*(\Sigma)$ to $A^*(Y(\Sigma))$ identifies $(c_{\overline X} \cdot c_{\overline X'})(\tau)$ with the intersection number $\deg([\overline X] \cdot [\overline X'] \cdot V(\tau))$.  By hypothesis, $\overline X$ meets $\overline X'$ properly in $Y(\Sigma)$, and the generic point of each component of the intersection lies in $T$, so we have the refined intersection
\[
\overline X \cdot \overline X' = \sum_Z i(Z, X \cdot X'; T) [\overline Z],
\] 
where the sum is taken over components $Z$ of $X \cap X'$.  Therefore,
\[
\deg([\overline X] \cdot [\overline X'] \cdot [V(\tau)]) = \sum_Z i(Z, X \cdot X'; T) \deg([\overline Z] \cdot [V(\tau)]).
\] 
Now, since $\Trop(X)$ and $\Trop(X')$ are subfans of $\Sigma$, the closures $\overline X$ and $\overline X'$ meet $V(\sigma)$ properly in $Y(\Sigma)$, for all $\sigma \in \Sigma$.  Furthermore, by hypothesis, $\overline X \cap V(\sigma)$ and $\overline X' \cap V(\sigma)$ meet properly in $V(\sigma)$ for all $\sigma$.  It follows that $Z$ meets all orbits of $Y(\Sigma)$ properly.  Therefore $\Trop(Z)$ is a union of cones of codimension $j + j'$ in $\Sigma$ (see Remark~\ref{rem:meeting-orbits}) and the intersection number $\deg([\overline Z] \cdot [V(\tau)])$ is equal to the tropical multiplicity $m_Z(\tau)$.
\end{proof}

The following proposition may be seen as a geometric explanation for the existence of well-defined stable tropical intersections, in the case of the trivial valuation; they are exactly the tropicalizations of intersections of generic translates in the torus.

\begin{prop} \label{prop:translation}
Suppose the valuation is trivial, and let $t$ be a general point in $T$.  Then $X$ meets $tX'$ properly and the tropicalization of the intersection scheme $X \cap tX'$ is exactly the stable tropical intersection $\Trop(X) \cdot \Trop(X')$.
\end{prop}

\begin{proof}
First, an easy Bertini argument shows that $X$ meets $tX'$ properly when $t$ is general.  Furthermore, since $Y(\Sigma)$ is smooth, subadditivity of codimension says that every component of $\overline X \cap t\overline X'$ has codimension at most $j + j'$.  Say $Z$ is a component of this intersection, and let $O_\sigma$ be the torus orbit in $Y(\Sigma)$ that contains the generic point of $Z$.  Now, $\overline X$ and $t \overline X'$ meet $O_\sigma$ properly, since their tropicalizations are subfans of $\Sigma$, and $\overline X \cap O_\sigma$ meets $t \overline X' \cap O_\sigma$ properly in $O_\sigma$, since $t$ is general.  Therefore $Z$ has codimension $j + j' + \codim O_\sigma$, and hence $O_\sigma$ must be the dense torus $T$.

Next, the Cohen-Macaulay locus in an excellent scheme is open \cite[Scholium 7.8.3(iv)]{EGA4.2}, and contains every generic point, so the Cohen-Macaulay loci of $X$ and $X'$ are open and dense. Therefore, since $t$ is general, the generic point of every component of $X \cap tX'$ must lie in the Cohen-Macaulay loci of both $X$ and $X'$.  Then, by  \cite[Example~8.2.7]{IT}, the intersection multiplicity $i(Z, X \cdot tX'; T)$ is equal to the length of the intersection of $X$ and $tX'$ along the generic point of $Z$.
Proposition~\ref{prop:cycles} then says that 
\[
i(\tau, \Trop(X) \cdot \Trop(X')) = \sum_Z \length( \cO_{Z, X \cap tX'}) \, m_Z(\tau).
\]
Now the right hand side is the multiplicity of $\tau$ in the tropicalization of the cycle $[X \cap tX']$, and the proposition follows, by Lemma~\ref{lem:linearity}.
\end{proof}

\begin{rem}
We apply this proposition in the proof of Lemma~\ref{lem:support}.  All that is needed for that application is the standard fact that the intersection multiplicity at a point of proper intersection is strictly positive.  In particular, the equality of the intersection multiplicity $i(Z, X \cdot tX'; T)$ with the length of the scheme-theoretic intersection of $X$ and $X'$ along $Z$ is not needed for the main lifting theorems.
\end{rem}

\noindent   In Section~\ref{sec:main}, we extend Proposition~\ref{prop:translation} to the general case, where the valuation may be nontrivial, as an application of tropical lifting theorems.  See Theorem~\ref{thm:translation}.

\section{The trivial valuation case} \label{sec:trivial val}

The first step in our approach to tropical lifting theorems is to understand how the intersection of $\Trop(X)$ and $\Trop(X')$ near a point $w$ relates to the intersection of the initial degenerations $X_w$ and $X'_w$.  If $w$ is in $N_G$ then the stars of $w$ in $\Trop(X)$ and $\Trop(X')$ are the tropicalizations of $X_w$ and $X'_w$, respectively, with respect to the trivial valuation, so this amounts to understanding how tropicalization relates to intersections when the tropicalizations meet properly and the valuation is trivial.  Our proof that tropicalization commutes with intersection in this case uses both the theory of extended tropicalizations as projections of nonarchimedean analytifications and a new result about the support of stable intersections of tropicalizations (Lemma~\ref{lem:support}).  We prove two necessary lemmas about the topology of extended tropicalizations in the general case, where the valuation is not necessarily trivial, since the arguments are identical to those in the case where the valuation is trivial.

\subsection{Topology of extended tropicalizations}

 The extended tropicalizations of $Y(\Sigma)$ and its closed subschemes were introduced by Kajiwara \cite{Kajiwara08}, and their basic properties were studied further in Section~3 of \cite{analytification}, to which we refer the reader for further details. We recall the definition and prove that the extended tropicalization of a closed subscheme $X$ in which the torus $T$ is dense is the closure of the ordinary tropicalization of $X \cap T$.  We also characterize the closure in the extended tropicalization $\bTrop(Y(\Sigma))$ of a cone $\sigma$ in $\Sigma$.

Let $\Sigma$ be a fan in $N_\RR$, and let $Y(\Sigma)$ be the associated toric variety.
Recall that each cone $\sigma$ in $\Sigma$ corresponds to an affine open subvariety $U_\sigma$ of $Y(\Sigma)$ whose coordinate ring is the semigroup ring $K[\sigma^\vee \cap M]$ associated to the semigroup of lattice points in the dual cone $\sigma^\vee$ in $M_\RR$.  Let $\bRR$ be the real line extended in the positive direction
\[
\bRR = \RR \cup \{+ \infty\},
\]
which is a semigroup under addition, with identity zero.  Then
\[
\bTrop(U_\sigma) = \Hom(\sigma^\vee \cap M, \bRR)
\]
is the space of all semigroup homomorphisms taking zero to zero, with its natural topology as a subspace of $\bRR^{\sigma^\vee \cap M}$.  For each face $\tau$ of $\Sigma$, let $N(\tau)$ be the real vector space
\[
N(\tau) = N_\RR / \Span(\tau).
\]
Then $\Trop(U_\sigma)$ is naturally a disjoint union of the real vector spaces $N(\tau)$, for $\tau \preceq \sigma$, where $N(\tau)$ is identified with the subset of semigroup homomorphisms that are finite exactly on the intersection of $\tau^\perp$ with $\sigma^\vee \cap M$.

The tropicalization $\bTrop(Y(\Sigma))$ is constructed by gluing the spaces $\bTrop(U_\sigma)$ for $\sigma \in \Sigma$ along the natural open inclusions $\bTrop(U_\tau) \subset \bTrop(U_\sigma)$ for $\tau \preceq \sigma$.  It is a disjoint union
\[
\bTrop(Y(\Sigma)) = \bigsqcup_{\sigma \in \Sigma} N(\sigma),
\]
just as $Y(\Sigma)$ is the disjoint union of the torus orbits $O_\sigma$.  Points in $N(\sigma)$ may be seen as weight vectors on monomials in the coordinate ring of $O_\sigma$, and the tropicalization of a closed subscheme $X$ in $Y(\Sigma)$ is defined to be
\[
\bTrop(X) = \bigsqcup_{\sigma \in \Sigma} \{ w \in N(\sigma) \ | \ (X \cap O_\sigma)_w \mbox{ is not empty}\}.
\]
In other words, $\bTrop(X)$ is the disjoint union of the tropicalizations of its intersections with the $T$-orbits in $Y(\Sigma)$. This space is compact when
$\Sigma$ is complete.

\begin{lem}\label{lem:extended-closure} 
If every component of $X$ meets the dense torus $T$, then $\bTrop(X)$ is the closure in $\bTrop(Y(\Sigma))$ of the ordinary tropicalization $\Trop(X \cap T)$.
\end{lem}

\begin{proof}
Since tropicalizations are invariant under extensions of valued fields \cite[Proposition~6.1]{analytification}, and the completion of the algebraically closed field $K$ is algebraically closed \cite[Proposition~3.4.1.3]{BGR84}, we may assume the field $K$ is complete with respect to its valuation.
 
The tropicalization map from $X(K)$ to $\bTrop(Y(\Sigma))$ extends to a proper continuous map on the nonarchimedean analytification $X^\an$ whose image is $\bTrop(X)$ \cite[Section~2]{analytification}.  Since the open subset $(X \cap T)^\an$ is dense in $X^\an$ \cite[Corollary~3.4.5]{Berkovich90} and maps onto $\Trop(X \cap T)$, the extended tropicalization $\bTrop(X)$ is contained in the closure of $\Trop(X \cap T)$.  The lemma follows, since $\bTrop(X)$ is closed.
\end{proof}

If $\sigma$ is a face of $\Sigma$ that contains another face $\tau$, then we write $\sigma_\tau$ for the image of $\sigma$ in $N(\tau)$.  It is a cone of dimension $\dim(\sigma) - \dim(\tau)$.  We will use the following lemma in the proof of Proposition~\ref{prop:proper at boundary}.

\begin{lem} \label{lem:boundary}
Let $\sigma$ and $\tau$ be faces of $\Sigma$, and let $\overline \sigma$ be the closure of $\sigma$ in $\bTrop(Y(\Sigma))$.  Then
\[
\overline \sigma \cap N(\tau) = \sigma_\tau
\]
if $\sigma$ contains $\tau$, and $\overline \sigma$ is disjoint from $N(\tau)$ otherwise.
\end{lem}

\begin{proof}
Let $v$ be a point in $N(\tau)$, and let $\pi: \bTrop(U_\tau) \rightarrow N(\tau)$ be the continuous map that restricts to the canonical linear projections from $N(\gamma)$ onto $N(\tau)$ for $\gamma \preceq \tau$.  If $v$ is not in $\sigma_\tau$, then the preimage under $\pi$ of $N(\tau) \smallsetminus \sigma_\tau$ is an open neighborhood of $v$ that is disjoint from $\sigma$, so $v$ is not in $\overline \sigma$.  Similarly, if $\sigma$ does not contain $\tau$ then $\pi^{-1}(N(\tau) )\smallsetminus \sigma$ is an open neighborhood of $v$ that is disjoint from $\sigma$.

It remains to show that $\overline{\sigma} \cap N(\tau)$ contains $\sigma_\tau$ in the case where $\sigma$ contains $\tau$.  Suppose $v$ is in $\sigma_\tau$ and let $w$ be a point in $\sigma$ that projects to $v$.  Let $w'$ be a point in the relative interior of $\tau$.  Then $w + \RR_{\geq 0} w'$ is a path in $\sigma$ whose limit is $v$, so $v$ is in $\overline \sigma$.
\end{proof}

\subsection{Tropical subadditivity}

Although not strictly necessary for the main results of the paper, it is helpful to know that codimension is subadditive for intersections of tropicalizations in $N_\RR$, so $\Trop(X)$ and $\Trop(X')$ never intersect in less than the expected dimension.  The proof is by a diagonal projection argument, similar to the proof of Lemma~\ref{lem:support}, below.

\begin{prop} \label{prop:tropical subadditivity}
Let $X$ and $X'$ be closed subschemes of pure codimension $j$ and $j'$, respectively, in $T$.  If $\Trop(X) \cap \Trop(X')$ is nonempty then it has codimension at most $j + j'$ at every point.
\end{prop}

\begin{proof}
Let $w$ be a point in $\Trop(X) \cap \Trop(X')$.  Then, in a neighborhood of $w$, the intersection can be identified with a neighborhood of zero in $\Trop(X_w) \cap \Trop(X'_w)$.  Therefore, replacing $X$ and $X'$ by their initial degenerations at $w$, we may assume the valuation is trivial, and it will suffice to show that the global codimension of $\Trop(X) \cap \Trop(X)$ is at most $j + j'$.

Suppose the valuation is trivial.  Let $T'$ be the quotient of $T \times T$ by the diagonal subtorus, with $\pi: X \times X' \rightarrow T'$ the induced projection, and $p: \Trop(X) \times \Trop(X') \rightarrow N'_\RR$ the tropicalization of $\pi$.  Since $\Trop(X)$ meets $\Trop(X')$, the point zero is in the image of $p$, which is exactly the tropicalization of the closure of the image of $\pi$.  

Let $y$ be a point in the image of $\pi$.  The fiber $\pi^{-1}(y)$ has dimension at least $\dim X + \dim X' - \dim T'$ and, since the valuation is trivial, its tropicalization is contained in $p^{-1}(0)$.  Since $p^{-1}(0)$ is naturally identified with $\Trop(X) \cap \Trop(X')$, it follows that the tropical intersection has dimension at least $\dim X' + \dim X - \dim T$, as required.
\end{proof}

\subsection{Lower bounds on multiplicities}

For the remainder of this section, we assume the valuation $\nu$ is trivial.  Let $X$ and $X'$ be closed subschemes of pure codimension $j$ and $j'$ in $T$, respectively.  After subdividing the tropicalizations, we choose a complete unimodular fan $\Sigma$ in $N_\RR$ such that each face of $\Trop(X)$ and each face of $\Trop(X')$ is a face of $\Sigma$.  We write $\overline X$ and $\overline X'$ for the closures of $X$ and $X'$ in the smooth complete toric variety $Y(\Sigma)$.

Our goal is to prove the following version of Theorem~\ref{basic} in the special case of the trivial valuation, with lower bounds on the multiplicities of the facets in $\Trop(X \cap X')$.  These lower bounds are extended to the general case in Section~\ref{sec:main}.

\begin{thm} \label{thm:trivial val}
Suppose $\nu$ is trivial and $\Trop(X)$ meets $\Trop(X')$ properly.  Then 
\[
\Trop(X \cap X') = \Trop(X) \cap \Trop(X'), 
\]
and the multiplicity of any facet $\tau$ is bounded below by the tropical intersection multiplicity
\[
m_{X \cap X'}(\tau) \geq i(\tau, \Trop(X) \cdot \Trop(X')).
\]
Furthermore, the tropical intersection multiplicity is equal to the weighted sum of algebraic intersection multiplicities
\[
i(\tau,\Trop(X)\cdot \Trop(X')) = \sum_{Z} i(Z,X \cdot X';Y(\Sigma)) \, m_Z(\tau),
\]
where the sum is over components $Z$ of $X \cap X'$ whose tropicalizations contain $\tau$.
\end{thm}

\noindent Given the compatibility of toric and stable tropical intersections in Proposition~\ref{prop:cycles}, the main step in the proof of Theorem~\ref{thm:trivial val} is showing that every component of $\overline X \cap \overline X'$ is the closure of a component of $X \cap X'$.  This step can fail spectacularly when $\Trop(X)$ and $\Trop(X')$ do not meet properly.  In such cases, $\overline X \cap \overline X'$ may contain components of larger than expected dimension, even if $X$ and $X'$ meet properly in T, and it is difficult to predict what $\Trop(X \cap X')$ will look like.

\begin{prop} \label{prop:proper at boundary}
Suppose $\Trop(X)$ meets $\Trop(X')$ properly, and let $\tau$ be a face of $\Sigma$.  Then 
\begin{enumerate}
\item The tropicalizations of $\overline X \cap O_\tau$ and $\overline X' \cap O_\tau$ meet properly in $N(\tau)$.
\item The subschemes $\overline X \cap V(\tau)$ and $\overline X' \cap V(\tau)$ meet properly in $V(\tau)$.
\item The closures $\overline X$ and $\overline X'$ meet properly in $Y(\Sigma)$, and every component of $\overline X \cap \overline X'$ is the closure of a component of $X \cap X'$.
\end{enumerate}
\end{prop}

\begin{proof}
First, if the tropicalizations of $\overline X \cap O_\tau$ and $\overline X'\cap O_\tau$ meet properly in $N(\tau)$, then the intersections themselves must meet properly in $O_\tau$, by part (1) of the foundational theorem in Section~\ref{sec:tropicalization}.  If this holds for all cones $\sigma$ containing $\tau$, then $\overline X \cap V(\tau)$ meets $\overline X'\cap V(\tau)$ properly in $V(\tau)$.  Therefore (1) implies (2).  

Now, since $\Trop(X)$ and $\Trop(X')$ are subfans of $\Sigma$, the closures $\overline X$ and $\overline X'$ meet all orbits of $Y(\Sigma)$ properly (see Section~\ref{sec:compatibility}).  Hence, if (2) holds then every component of $\overline X \cap \overline X'\cap V(\tau)$ has codimension $j + j'$ in $V(\tau)$.  In particular, if $O_\tau$ is not the dense torus $T$, then $V(\tau)$ does not contain a component of $\overline X \cap \overline X'$.  If this holds for all $\tau$, then every component of $\overline X \cap \overline X'$ meets the dense torus $T$.  This shows that (2) implies (3).  We now prove (1).
  
The tropicalization of $\overline X \cap O_\tau$ is the intersection of the extended tropicalization $\bTrop(\overline X)$ with $N(\tau)$ and, by Lemma~\ref{lem:extended-closure}, this extended tropicalization is the closure of $\Trop(X)$.  Therefore, by Lemma~\ref{lem:boundary}, $\Trop(\overline X \cap O_\tau)$ is the union of the projected faces $\sigma_\tau$ such that $\sigma$ is in $\Trop(X)$ and contains $\tau$.  Similarly, $\Trop(\overline X' \cap O_\tau)$ is the union of the projected faces $\sigma'_\tau$ such that $\sigma'$ is in $\Trop(X')$ and contains $\tau$.  Since the intersection of $\sigma_\tau$ and $\sigma'_\tau$ is exactly $(\sigma \cap \sigma')_\tau$, the codimension of $\sigma_\tau \cap \sigma'_\tau$ in $N(\tau)$ is equal to the codimension of $\sigma \cap \sigma'$ in $N_\RR$.  In particular, if $\Trop(X)$ and $\Trop(X')$ meet properly at $\tau$, then $\Trop(\overline X \cap O_\tau)$ and $\Trop(\overline X' \cap O_\tau)$ meet properly in $N(\tau)$, as required.
\end{proof}

The following lemma says that if $\Trop(X)$ meets $\Trop(X')$ properly then the support of their stable tropical intersection is equal to their set-theoretic intersection.

\begin{lem} \label{lem:support}
Suppose $\Trop(X)$ and $\Trop(X')$ meet properly and let $\tau$ be a facet of their intersection.  Then the tropical multiplicity $i(\tau, \Trop(X) \cdot \Trop(X'))$ is strictly positive.
\end{lem}

\begin{proof}
First, we reduce to the case where $X$ and $X'$ have complementary dimension.  Recall that the tropical multiplicity $m(\sigma)$ of a facet $\sigma$ in $\Trop(X)$ is equal to the algebraic intersection number $\deg(\overline X \cdot V(\sigma))$ in $Y(\Sigma)$, as discussed in Section~\ref{sec:compatibility}, and the intersection number $\deg(\overline X \cdot \overline X' \cdot V(\tau))$ in $Y(\Sigma)$ is equal to the intersection number
\[
\deg\big((\overline X \cdot V(\tau)) \cdot (\overline X' \cdot V(\tau))\big)
\]
in $V(\tau)$ \cite[Example~8.1.10]{IT}, because $V(\tau)$ is smooth.  Because $\Trop(X)$ and $\Trop(X')$ are subfans of $\Sigma$, the closures $\overline X$ and $\overline X'$ in $Y(\Sigma)$ meet all torus orbits properly, and it follows that the tropical intersection multiplicity $i(\tau, \Trop(X) \cdot \Trop(X'))$ is equal to the tropical intersection multiplicity at zero of $\Trop(\overline X \cdot O_\tau)$ and $\Trop(\overline X' \cdot O_\tau)$ inside $N(\tau)$.  By Proposition~\ref{prop:proper at boundary}, the intersection cycles $\overline X \cdot O_\tau$ and $\overline X' \cdot O_\tau$ have complementary dimension in $O_\tau$ and their tropicalizations meet properly in $N(\tau)$.  The same is then true for each of the reduced components of $\overline X \cdot O_\tau$ and $\overline X' \cdot O_\tau$, and it suffices to prove the proposition in this case.

Assume $X$ and $X'$ have complementary dimension.  Since their tropicalizations meet properly, by hypothesis, their intersection must be the single point zero.  Consider the quotient $T'$ of $T \times T$ by its diagonal subtorus, and the induced projection
\[
\pi: X \times X' \rightarrow T'.
\]
By Proposition~\ref{prop:translation}, the tropical intersection number $\Trop(X) \cdot \Trop(X')$ is equal to the length of $X \cap tX'$ for a general point $t \in T$, which is equal to the length of the general fiber of $\pi$.  In particular, to prove the proposition, it is enough to show that $\pi$ is dominant.  Now, since the valuation is trivial, the tropicalization of any fiber of $\pi$ is contained in the preimage of zero under the induced projection
\[
\Trop(X) \times \Trop(X') \rightarrow N'_\RR,
\]
which is the single point zero, by hypothesis.  In particular, every nonempty fiber of $\pi$ is zero-dimensional, and hence $\pi$ is dominant, as required.
\end{proof}

We can now easily prove the desired result.

\begin{proof}[Proof of Theorem~\ref{thm:trivial val}]
Let $\tau$ be a facet of $\Trop(X) \cap \Trop(X')$. By Propositions~\ref{prop:cycles} and \ref{prop:proper at boundary}, the tropical intersection multiplicity $i(\tau, \Trop(X) \cdot \Trop(X'))$ is a weighted sum over components of $X \cap X'$
\[
i(\tau, \Trop(X) \cdot \Trop(X')) = \sum_Z i(Z, X \cdot X';T) m_Z(\tau).
\]
By Lemma~\ref{lem:support}, this multiplicity is strictly positive, so there is at least one component $Z$ of $X \cap X'$ such that $\Trop(Z)$ contains $\tau$.  This proves that $\tau$ is contained in $\Trop(X \cap X')$, and hence $\Trop(X \cap X')$ is equal to $\Trop(X) \cap \Trop(X')$, set-theoretically.

Finally, to see the inequality on multiplicities, recall that the multiplicity of $\tau$ in $\Trop(X \cap X')$ is
\[
m_{X \cap X'}(\tau) = \sum_Z \length_{\cO_{T,Z}}(\cO_{X \cap X',Z}) m_{Z}(\tau),
\]
by Lemma~\ref{lem:linearity}.  The inequality $m_{X \cap X'}(\tau) \geq i(\tau, \Trop(X) \cdot \Trop(X'))$ follows by comparing the two summation formulas term by term, because $\length_{\cO_{T,Z}}(\cO_{X \cap X',Z})$ is at least as large as the intersection multiplicity $i(Z, X \cdot X';T)$ \cite[Proposition~8.2]{IT}.
\end{proof}

\smallskip

\section{Geometry over valuation rings of rank 1} \label{sec:dim}

In \cite{Nagata66}, Nagata began extending certain results from dimension theory to non-noetherian rings.  Here we continue in Nagata's spirit with a view toward tropical geometry, focusing on valuation rings of rank 1 and showing that these rings have many of the pleasant properties of regular local rings.  

Recall that a valuation ring of rank 1 is an integral domain $R$ whose field of fractions $K$ admits a nontrivial valuation $\nu: K^* \rightarrow \RR$ such that the nonzero elements of $R$ are exactly the elements of $K^*$ of nonnegative valuation.  It is a local ring with exactly two prime ideals, the zero ideal and the maximal ideal.  If $K$ is algebraically closed or, more generally, if the valuation is not discrete, then $R$ is not noetherian, because the maximal ideal is not finitely generated.

Throughout this section, $R$ is a valuation ring of rank 1, with fraction field $K$ and residue field $k$.  In this section only, we allow the possibility that $K$ is not algebraically closed, because the results of this section hold equally for nonclosed fields and the greater generality creates no additional difficulties.  We fix
\[
S = \Spec R,
\]
and use calligraphic notation for schemes over $S$.  If $\cX$ is a scheme over $S$, we write $\cX_K = \cX \times_S \Spec K$ for the generic fiber and $\cX_k= \cX \times_S \Spec k$ for the special fiber.  

A module is flat over $R$ if and only if it is torsion-free \cite[Exercise~10.2]{Matsumura89}, so an integral scheme $\cX$ is flat over $S$ if and only if it is dominant.  Furthermore, a scheme that is flat and locally of finite type over an integral scheme is necessarily locally of finite presentation \cite[Theorem~3.4.6 and Corollary 3.4.7]{RaynaudGruson71}.  In particular, if $\cX$ is an integral scheme that is surjective and locally of finite type over $S$ then the special fiber $\cX_k$ is of pure dimension equal to the dimension of the generic fiber $\cX_K$.  This follows from \cite[Theorem~12.1.1(i)]{EGA4.3}, applied to the generic point of each component of the special fiber.  We will use these technical facts repeatedly throughout the section, without further mention.  

Dimension theory over valuation rings of rank 1 involves some subtleties.  For instance, if $R$ is not noetherian then the formal power series ring $R[[x]]$ has infinite Krull dimension \cite{Arnold73}.  Because dimensions are well-behaved under specialization over $R$, such pathologies are always visible in at least one of the fibers.  For $\cX = \Spec R[[x]]$, the special fiber is $\Spec k[[x]]$, which is one dimensional, but $R[[x]] \otimes K$ is the subring of $K[[x]]$ consisting of formal power series whose coefficients have valuations bounded below.  This ring, and hence the generic fiber of $\Spec R[[x]]$, is infinite dimensional.  Nevertheless, for a scheme of finite type over $S$, the generic fiber and special fiber are schemes of finite type over $K$ and $k$, respectively, and such pathologies do not occur.

Many of the results of this section have natural generalizations to more general base schemes.  See Appendix~\ref{app:closed-pts} for details.

\subsection{Subadditivity and lifting over valuation rings of rank 1}  

Serre famously proved that codimension is subadditive under intersections in regular schemes \cite[Theorem~V.3]{Serre65}.  In other words, for any irreducible closed subschemes $X$ and $X'$ of a regular scheme $Y$, and any irreducible component $Z$ of $X \cap X'$, we have 
\[
\codim_Y Z \leq \codim_Y X + \codim_Y X'.
\]
The same is then necessarily true for schemes smooth over $Y$.  We extend Serre's theorem to smooth schemes over rank 1 valuation rings, as follows.  Recall that we have fixed $S = \Spec R$, where $R$ is a valuation ring of rank 1.

\begin{thm} \label{thm:int reg}
Let $\cY$ be smooth over $S$. Then codimension is subadditive under intersection in $\cY$.
\end{thm}

\noindent In this respect, valuation rings of rank 1 behave like regular local rings.

\begin{rem}
Subadditivity of codimension is often used to deform points in families using dimension counting arguments.  Such techniques are essential, for instance, in the theory of limit linear series developed by Eisenbud and Harris \cite{EisenbudHarris86}.  Here, we use subadditivity of codimension to lift intersection points from the special fiber to the generic fiber in schemes of finite type over valuation rings of rank 1.\footnote{The assumption that schemes are of finite type over $S$ is crucial for the main results of this section.  In the tropical setting, if $\nu$ is nontrivial and $w$ is not in $N_G$, then closed points of $X_w$ never lift to closed points of $X$.  However, this does not contradict Theorem~\ref{thm:closed-pts} because in this case $\cX^w$ is not of finite type.  See Appendix~\ref{app:finite type}.}
\end{rem}

As an application of Theorem~\ref{thm:int reg} we prove the following. 

\begin{thm} \label{thm:lift to K}
Let $\cY$ be smooth over $S$ with closed subschemes $\cX$ and $\cX'$ of pure codimension $j$ and $j'$, respectively.  Suppose the special fibers $\cX_k$ and $\cX'_k$ intersect in codimension $j + j'$ at a point $x$ in $\cY_k$.  Then 
\begin{enumerate}
\item There is a point in $\cX_K \cap \cX'_K$ specializing to $x$.  
\item If $x$ is closed then there is a closed point in $\cX_K \cap \cX'_K$ specializing to $x$. More generally, if $k(x)$ denotes the residue field of $x$, there is a point $x' \in \cX_K \cap \cX'_K$ specializing to $x$ and satisfying
$$\trdeg k(x')/K = \trdeg k(x)/k.$$
\end{enumerate}
\end{thm}

\noindent In Section~\ref{subsec:tor} we also prove a principle of continuity for intersection multiplicities in families over S.

\begin{rem}
The theorem reduces easily to the case where $\cX$ and $\cX'$ are reduced and irreducible and then, using subadditivity, it is not difficult to show that both must meet the generic fiber $\cY_K$ and are therefore flat over $S$.  Our methods also give a more general lifting theorem for intersections of flat subschemes over an arbitrary base scheme (Theorem~\ref{thm:lift general}), even when subadditivity fails.  Here we proceed through subadditivity because it eliminates the need for a flatness hypothesis and makes the lifting arguments particularly transparent, especially for those familiar with similar arguments over regular local rings.
\end{rem}

\subsection{A principal ideal theorem}

Recall that Krull's principal ideal theorem, translated into geometric terms, says that every component of a Cartier divisor on a noetherian scheme has codimension $1$.  Here we prove the following generalization to valuation rings of rank 1.

\begin{thm}\label{thm:val-pit} 
Let $\cX$ be of finite type over $S$, and let $\cZ$ be a locally principal closed subscheme of $\cX$. Then every irreducible component of $\cZ$ has codimension at most $1$ in each component of $\cX$ that contains it.
\end{thm}

\noindent The same is not true for valuation rings that are not of rank 1.  If $A$ is such a ring then $\Spec A$ has principal subschemes of codimension greater than 1.  A key step in the proof of the theorem is the following technical proposition, which we prove by noetherian approximation.

\begin{prop}\label{prop:krull} Let $\cX$ be irreducible, locally of finite type, and flat over $S$.  Suppose that $\cD$ is a locally principal closed subscheme in $\cX$ that does not meet the generic fiber $\cX_K$.  Then every irreducible component of $\cD_k$ is an irreducible component of $\cX_k$.
\end{prop}

\begin{proof}
Since $\cX$ is flat and locally of finite type over $S$, it is locally of finite presentation over $S$. The question is local, so we may assume that $\cX$ is affine, and in particular of finite presentation over $S$.  Therefore, there exists a finitely generated subring $R' \subset R$, with models $X'$ and $D'$ over $S' = \Spec R'$ of $\cX$ and $\cD$, respectively.  Then $S'$ is irreducible by construction, and we may assume $X'$ is irreducible as well.

Let $s'$ be the image in $S'$ of the closed point in $S$.  Since $\cX_K$ and $\cX_k$ are both pure of the same dimension $d$, the fibers of $X'$ over the generic point and $s'$ are also pure of dimension $d$. Similarly, $\cD_k$ is pure of dimension $d$ if and only if $D'_{s'}$ is, so to prove the proposition it is enough to show that $D'_{s'}$ is pure of dimension $d$.

Let $Z$ be a component of $D'$ that meets the fiber over $s'$.  We will show that $Z_{s'}$ has pure dimension $d$.  Let $s''$ be the image in $S'$ of the generic point of $Z$.  Then $s''$ is a specialization of the generic point of $S'$ and specializes to $s'$, so upper semicontinuity of fiber dimension \cite[Theorem 13.1.3]{EGA4.3} implies that $X'_{s''}$ has pure dimension $d$.  We claim that $Z_{s''}$ is a component of $X'_{s''}$.  To see this, note that $s''$ is not the generic point of $S'$, since $\cD$ does not meet the generic fiber of $\cX$, and hence every component of $X'_{s''}$ has codimension at least 1 in $X'$.  Since $X'$ is noetherian, Krull's principal ideal theorem says that every component of $D'$ has codimension at most 1 in $X'$, and hence $Z_{s''}$ is a component of $X'_{s''}$, as claimed.  In particular, $Z_{s''}$ has pure dimension $d$.  Since $s''$ specializes to $s'$ and $X'_{s'}$ has pure dimension $d$, it follows by upper semicontinuity that $Z_{s'}$ has pure dimension $d$, as required.
\end{proof}

The following two lemmas are special cases of a more general altitude formula over valuation rings of finite rank \cite[Theorem~2]{Nagata66}.  These codimension formulas are used in the proof of the principal ideal theorem, and again in the proof of subadditivity of codimension over valuation rings of rank 1.  We include proofs for the reader's convenience.

\begin{lem} \label{lem:nagata-easy1}
Suppose $\cX$ is irreducible, finite type, and flat over $S$ with $\cZ \subset \cX_k$ an irreducible closed subset of the special fiber.  Then
\[
\codim_\cX \cZ = \codim_{\cX_k} \cZ + 1.
\]
\end{lem}

\begin{proof}
A maximal chain of irreducible closed subsets between $\cZ$ and $\cX_k$ can be extended by adding $\cX$ at the end, so $\codim_\cX \cZ$ is at least $\codim_{\cX_k} \cZ + 1$.  We now show that $\codim_{\cX} \cZ$ is at most $\codim_{\cX_k} \cZ + 1$.

If $W \subset W'$ is a strict inclusion of irreducible closed subsets of $\cX$ that meet the special fiber then the dimension of $W_k$ is less than or equal to the dimension of $W'_k$.  This inequality is strict if $W$ and $W'$ are both contained in the special fiber.  Similarly, if $W$ and $W'$ both meet the generic fiber then $\dim W_K$ is less than $\dim W'_K$, and the dimensions of the special fibers are equal to the dimensions of the respective generic fibers, by flatness.  Therefore, if $\dim W_k$ is equal to $\dim W'_k$ then $W$ is contained in the special fiber and $W'$ is not.  This can happen at most once in any chain of inclusions, so any chain between $\cZ$ and $\cX$ has length at most $\dim \cX_k - \dim Z + 1$, as required.
\end{proof}

\begin{lem}\label{lem:nagata-easy} Suppose $\cX$ is irreducible, finite type, and flat over $S$, and let $\cZ \subset \cZ'$ be irreducible closed subsets. 
Then
\[
\codim_\cX \cZ=\codim_\cX \cZ' + \codim_{\cZ'} \cZ.
\]
\end{lem}

\begin{proof} 
If $\cZ$ meets the generic fiber of $\cX$ then this is a classical formula for codimension of varieties over $K$.  Suppose $\cZ$ is contained in the special fiber.  If $\cZ'$ meets the generic fiber then $\codim_\cX \cZ' = \dim \cX_K - \dim \cZ'_K$, and $\codim_{\cZ'} \cZ = \dim \cZ'_k - \dim \cZ + 1$, by Lemma~\ref{lem:nagata-easy1}.  Now $\dim \cX_K$ and $\dim \cZ'_K$ are equal to $\dim \cX_k$ and $\dim \cZ'_k$, respectively, so adding these two equations gives
\[
\codim_\cX \cZ' + \codim_{\cZ'} \cZ = \dim \cX_k - \dim \cZ + 1, 
\]
and Lemma~\ref{lem:nagata-easy1} says that the right hand side is equal to $\codim_\cX Z$.  The proof in the case where $\cZ'$ is contained in the special fiber is similar.
\end{proof}

We now prove the principal ideal theorem over valuation rings of rank 1.

\begin{proof}[Proof of Theorem~\ref{thm:val-pit}]  
We may assume that $\cX$ is integral, and since the statement is local about the generic points of $\cZ$, we may assume that $\cZ$ is irreducible.  If $\cX$ is supported in the special fiber, the theorem reduces to the classical principal ideal theorem over $k$.  Otherwise $\cX$ is dominant, and hence flat, over $S$.  If $\cZ$ meets the generic fiber $\cX_K$, then the theorem follows from the classical principal ideal theorem for $\cX_K$. On the other hand, if $\cZ$ is contained in the special fiber, then it must be a union of components of the special fiber, by Proposition~\ref{prop:krull}, and hence has codimension 1 in $X$, by Lemma~\ref{lem:nagata-easy1}.
\end{proof}

We apply the principal ideal theorem to prove the following result on lifting closed points in the special fiber to closed points in the generic fiber.  In fact, we prove a more general result, for points of arbitrary transcendence degree.  Our argument is in the spirit of Katz's proof of \cite[Lemma~4.15]{Katz09}.

\begin{thm}\label{thm:closed-pts}
Let $\cX$ be an irreducible scheme locally of finite type over $S$, and let $x$ be a closed point of the special fiber $\cX_k$.  If $x$ is in the closure of $\cX_K$ then the set of closed points $x'$ in $\cX_K$ specializing to $x$ is Zariski dense in $\cX_K$. More generally, if $x$ is not necessarily closed in $\cX_k$, and $k(x)$ denotes the residue field of $x$, we can choose $x'$ to satisfy the identity 
$$\trdeg k(x')/K = \trdeg k(x)/k,$$
and again the choices of $x'$ are Zariski dense in $\cX_K$.
\end{thm}

\begin{proof}
Note that the statement for closed points is a special case of the statement involving transcendence degrees.  We may assume $\cX$ is integral and in particular flat over $S$, and also affine.

By hypothesis, $x$ is a specialization of the generic point of $\cX$.  Observe that given $x' \in \cX_K$ specializing to $x$, the closure of $x'$ is flat over $S$, so we obtain the inequality $\trdeg k(x')/K \geq \trdeg k(x)/k$.  Now, if $\dim \cX_K = \trdeg k(x)/k$, then we may take $x'$ to be the generic point of $\cX_K$.  We thus proceed by induction on $d:=\dim \cX_K-\trdeg k(x)/k$.

Suppose now $d>0$, and note that this implies that $x$ is not the generic point of any component of $\cX_k$.   Let $W$ be a closed subset properly contained in $\cX_K$; by affineness, we can choose a regular function $f$ on $\cX$ that vanishes at $x$, but not on any component of $W \cup \cX_k$. Then let $\cD \subset \cX$ be the principal subscheme cut out by $f$.  Passing to a smaller neighborhood of $x$, if necessary, we may assume every component of $\cD$ contains $x$.  By Proposition~\ref{prop:krull}, some component $\cZ$ of $\cD$ meets the generic fiber, and hence $\dim \cZ_K - \trdeg k(x)/k = d-1$.  By induction, $\cZ_K$ contains a point $x'$ specializing to $x$, not contained in $W$, and with $\trdeg k(x')/K=\trdeg k(x)/k$.  We conclude the desired statement.
%
%
\end{proof}

\begin{rem} \label{rem:second gap}
For the initial degeneration of a closed subscheme of a torus $T$ over $K$ associated to a weight vector $w \in N_G$, in the special case where $x$ is closed, Theorem~\ref{thm:closed-pts} says that every point in $X_w(k)$ lifts to a point in $X(K)$, and the set of such lifts is Zariski dense.  In particular, the theorem gives an algebraic proof of surjectivity of tropicalization, along the lines suggested by Speyer and Sturmfels in \cite{SpeyerSturmfels04}, as well as a proof of the density of tropical fibers stated as Theorem~4.1 and Corollary~4.2 in \cite{tropicalfibers}.  A gap in the proof of the former paper is explained in a footnote on the first page of the latter.  The proof of Theorem~4.1 in the latter also contains a serious error, discovered by W. Buczynska and F. Sottile, which is explained and corrected in \cite{TropicalFibers-Correction}.
\end{rem}

\begin{rem}
When applied to non-closed points, Theorem~\ref{thm:closed-pts} says that every closed subvariety of $X_w$ is a component of the special fiber of the closure in $\cX^w$ of some closed subvariety of $X$.  So, at least in this weak sense, every curve in $X_w$ lifts to a curve in $X$, every surface in $X_w$ lifts to a surface in $X$, and so on.
\end{rem}

\subsection{Proofs of subadditivity and lifting}  \label{sec:subadd}

We now use the principal ideal theorem to prove subaddivity of codimension under intersection and deduce a lifting theorem for proper intersections, in smooth schemes over a valuation ring of rank 1. Our proof follows the classical argument for smooth varieties, which is simpler than Serre's proof for regular schemes.

\begin{rem}Codimension is not subadditive under intersection in smooth schemes over valuation rings of rank greater than 1, as shown by the following example.  Nevertheless, this failure of subadditivity can be understood and controlled, and proper intersections in special fibers still lift to more general fibers over arbitrary valuation rings and, more generally, for flat families of subschemes in a smooth scheme over an arbitrary base.  See Appendix~\ref{app:closed-pts}.
\end{rem}

\begin{ex} \label{ex:not subadd}
Let $A$ be a valuation ring of finite rank $r$ greater than 1, and let $a$ be an element of $A$ that generates an ideal of height $r$.  Then $A[x]$ has Krull dimension $r+1$, by \cite[Theorem~2]{Nagata66}.   The ideals of $A[x]$ generated by $x$ and $x-a$ each have height $1$, but the ideal $(x,x-a)=(x,a)$ has height $r+1$, which is greater than 2.
\end{ex}

Our proof of Theorem \ref{thm:int reg} involves a reduction to the diagonal and requires the following lemma on dimensions of fiber products over $S$.

\begin{lem}\label{lem:prod-cod} 
Let $\cY$ be irreducible and finite type over $S$, and let $\cX$ and $\cX'$ be irreducible closed subschemes.  Then for every irreducible component $\cZ$ of $\cX \times_S \cX'$, 
\[
\codim_{\cY \times_S \cY} \cZ \leq \codim_\cY \cX + \codim_\cY \cX'.
\]
Furthermore, the inequality is strict if and only if $\cY$ is dominant over $S$ and $\cX$ and $\cX'$ are both contained in the special fiber.
\end{lem}

\begin{proof}
If $\cX$ and $\cX'$ both meet the generic fiber then they are flat over $S$.  In this case, $\cZ$ must also meet the generic fiber, and the proposition holds with equality, by classical dimension theory over $K$.  Similarly, if $\cY$ is not dominant over $S$, then the proposition holds with equality by dimension theory over $k$.

Suppose $\cY$ is dominant and $\cX$ is contained in the special fiber.  Then $\cZ$ is also contained in the special fiber.  If, furthermore, $\cX'$ is contained in the special fiber, then $\cZ$ has dimension $\dim \cX + \dim \cX'$.  Therefore, by Lemma~\ref{lem:nagata-easy1},
\[
\codim_{\cY \times_S \cY} \cZ = 2 \dim \cY_K - \dim \cX - \dim \cX' + 1.
\]
The right hand side is equal to $\codim_\cY \cX + \codim_\cY \cX' -1$, so the inequality is strict by exactly one in this case.  A similar argument shows that the proposition holds with equality if $\cX'$ meets the generic fiber.
\end{proof}

We will prove subadditivity of codimension using the previous proposition and a reduction to the diagonal.
 
\begin{lem}\label{lem:smooth-lci}
If $\cY$ is smooth over $S$ then the diagonal $\Delta$ in $\cY \times_S \cY$ is a local complete intersection subscheme.
\end{lem} 

\noindent Here, by local complete intersection we mean that the number of local generators for $\cI_\Delta$ is equal to the codimension.
 
\begin{proof} 
First note that the diagonal morphism is locally of finite presentation, by \cite[Corollary~1.4.3.1]{EGA4.1}.  The lemma follows, since a locally finite presentation immersion of a smooth scheme in another smooth scheme over an arbitrary base is a local complete intersection.  See Proposition~7 of Section~2.2 in \cite{BLR90}.
\end{proof}

We now proceed with the proof of subadditivity of codimension and lifting of proper intersections in the special fiber, over valuation rings of rank 1.

\begin{proof}[Proof of Theorem \ref{thm:int reg}]
The intersection $\cX \cap \cX'$ can be realized as the intersection of $\cX \times_S \cX'$ with the diagonal in $\cY \times_S \cY'$.  Now, the codimension of any component $\cZ$ of $\cX \times_S \cX'$ is at most
\[
\codim _{\cY \times_S \cY} \cZ \leq \codim _\cY \cX + \codim_\cY \cX',
\]
by Lemma~\ref{lem:prod-cod}.  Then the principal ideal theorem (Theorem~\ref{thm:val-pit}) and Lemma~\ref{lem:nagata-easy} together imply that codimension can only decrease when intersecting with a local complete intersection subscheme of $\cY \times_S \cY$.  Lemma~\ref{lem:smooth-lci} says that the diagonal is a local complete intersection, and the theorem follows.
\end{proof}

\begin{proof}[Proof of Theorem \ref{thm:lift to K}]
Let $\cZ$ be a component of $\cX \cap \cX'$ that contains $x$.  By Theorem~\ref{thm:int reg}, 
\[ 
\codim_\cY \cZ \leq \codim_\cY \cX + \codim_\cY \cX'.
\]
The codimension of the special fiber $\cZ_k$ in $\cY$ is equal to $\codim_\cY \cX + \codim_\cY \cX' + 1$, by Lemma~\ref{lem:nagata-easy1}, so $\cZ$ must meet the generic fiber.  Therefore, the generic point of $\cZ_K$ is a point of $\cX_K \cap \cX'_K$ specializing to $x$.  If $x$ is closed in its fiber, we can then find a closed point of $\cZ_K$ specializing to $x$, by Theorem~\ref{thm:closed-pts}, and similarly for the assertion on transcendence degrees.
\end{proof}

\subsection{Intersection multiplicities over valuation rings of rank 1}\label{subsec:tor}

The principle of continuity says that intersection numbers are constant when cycles vary in flat families.  See \cite[Section~10.2]{IT} for a precise statement and proof when the base is a smooth variety.  For applications to tropical lifting with multiplicities, we need to apply a principle of continuity over the spectrum $S$ of a possibly non-noetherian valuation ring of rank 1.  Lacking a suitable reference, we include a proof.
 
\begin{defn} Suppose $Y$ is smooth over a field $k$, and
let $X$ and $X'$ be closed subschemes of $Y$ whose intersection $X \cap X'$ is finite. 
Then the \textbf{intersection number}
$i(X \cdot X'; Y)$ of $X$ and $X'$ in $Y$ is
$$\sum_{i=0}^{\dim Y} (-1)^i
\dim_k \Tor^i_{\cO_{Y}}(\cO_{X},\cO_{X'}).$$
\end{defn}
 
\noindent In other words, the intersection number is the sum of the local intersection multiplicities at the finitely many points of $X \cap X'$, weighted by degrees of extension fields,
\[ 
i(X \cdot X';Y)=  \sum_{P \in X \cap X'} [k(P):k] i(P,X \cdot X';Y).
\]
While we are primarily interested in the case where $X$ and $X'$ have complementary dimension, so the finiteness of $X \cap X'$ means that $X$ and $X'$ meet properly in $Y$, this hypothesis is not technically necessary, since the local intersection multiplicities vanish when the intersection is not proper \cite[Theorem~V.C.1]{Serre65}.
 
\begin{thm}\label{thm:int-mults} Let $\cY$ be smooth and quasiprojective over $S$, and let $\cX$ and $\cX'$ be closed subschemes of $\cY$ that are flat over $S$. If $\cX \cap \cX'$ is finite over $S$ then 
\[
i(\cX_K \cdot \cX'_K; \cY_K)=i(\cX_k \cdot \cX'_k; \cY_k).
\]
\end{thm}

\noindent In other words, intersection numbers are invariant under specialization over $S$.  We now give a direct proof of this equality, by first establishing coherence properties for the structures sheaves of these schemes and then using free resolutions to compute the $\Tor$ groups.  An alternative approach, using noetherian approximation and specialization properties of intersection theory over a noetherian base, is sketched in Remark~\ref{rem:noetherian-mults}.
 
In general, the structure sheaf of a non-noetherian scheme is not necessarily coherent.  The following proposition asserts that such pathologies do not occur on the schemes we are considering.  
 
\begin{prop}\label{prop:coherence} Let $\cX$ be locally of finite 
presentation over $S$.  Then $\cO_{\cX}$ is coherent.
\end{prop}
 
\begin{proof} 
We first claim that $\cO_{\cX}$ is coherent in the special case where $\cX$ is the affine space $\Spec R[t_1,\dots,t_n]$.  Since the question is local, it is enough to show that for any affine subset $\cU$ of $\cX$, and any homomorphism $f:\cO_{\cX}^m|_{\cU} \to \cO_{\cX}|_{\cU}$, the kernel of $f$ over $\cU$ is finitely generated. Now, the image of $f$ is finitely generated, and is the quasicoherent ideal sheaf $\tilde{I}$ associated to some ideal $I$ in $\cO_{\cX}(\cU)$.  The ideal $I$ is torsion-free and hence flat over $S$, so it is finitely presented by \cite[Theorem 3.4.6]{RaynaudGruson71}.  Therefore, by \cite[Theorem~2.6]{Matsumura89}, the kernel of the given presentation $\cO_{\cX}(\cU)^m \to I$ is finitely generated, which proves the claim.
 
For the general case,  since coherence is local, we may again assume that $\cX$ is affine, so $\cX=\Spec A[t_1,\dots,t_n]/I$ for some finitely generated ideal $I$.  The same argument as in the special case above shows that $I$ is finitely presented.  Therefore $\tilde{I}$ is coherent on $\Spec R[t_1,\dots,t_n]$, and hence $\cO_{\cX}$ is coherent \cite[Chapter~0, 5.3.10]{EGA1}.
\end{proof}
 
\begin{proof}[Proof of Theorem \ref{thm:int-mults}]
Since $\cY$ is quasiprojective and $\cX \cap \cX'$ is finite over $S$, there is an affine open subset of $\cY$ containing the intersection, and hence we may assume $\cY$ is affine.  Now $\cO_\cY$ is coherent, by Proposition~\ref{prop:coherence}, and $\cO_\cX$ and $\cO_{\cX'}$, being flat and finitely generated, are finitely presented and hence coherent as $\cO_\cY$-modules.  Let $\cP_\bullet$ be a resolution of $\cO_\cX$ by free $\cO_\cY$-modules of finite rank.  Then, since $\cO_\cX$ is flat over $S$, the restrictions $\cP_\bullet \otimes K$ and $\cP_\bullet \otimes k$ to the generic and special fibers remain exact, giving free resolutions of $\cO_{\cX_K}$ and $\cO_{\cX_k}$, respectively.  Therefore, not only does the homology of
\[
\cQ_\bullet = \cP_\bullet \otimes_{\cO_{\cY}} \cO_{\cX'}
\]
 compute $\Tor_\cY(\cX, \cX')$, but also the homology of the base changes $\cQ_\bullet \otimes K$ and $\cQ_\bullet \otimes k$ computes $\Tor_{\cY_K}(\cX_K, \cX'_K)$ and $\Tor_{\cY_k}(\cX_k, \cX'_k)$, respectively.

Now, $\Tor^i_\cY(\cX, \cX')$ is coherent, and supported on $\cX \cap \cX'$ \cite[6.5.1]{EGA3.2}.  It follows from \cite[Proposition 1.4.7]{EGA4.1} that the push forward of a quasicoherent, finitely presented module under a finite morphism of finite presentation is still quasicoherent and finitely presented.  Applying this to $\pi: \cX \cap \cX' \rightarrow S$ it follows that the push forward $\pi_* \Tor^i_\cY(\cX, \cX')$ is coherent, and since $\pi$ is affine, this push forward can be computed as the homology of the complex
\[
\cL_\bullet = \pi_* \cQ_\bullet.
\]
As above, since $\cO_\cX$ is flat over $S$, the homology of $\cL_\bullet \otimes K$ and $\cL_\bullet \otimes k$ compute $\Tor^i_{\cY_K}(\cX_K, \cX_K')$ and $\Tor^i_{\cY_k}(\cX_k, \cX_k')$, respectively.  Note also that since $\cO_{\cX'}$ is flat over $S$, the terms of $\cQ_\bullet$ and hence $\cL_\bullet$ are flat $\cO_S$-modules.  

Since the homology of $\cL_\bullet$ is coherent, there is a quasi-isomorphic bounded below complex $\cM_\bullet$ of free $\cO_S$-modules of finite rank, and since $\cL_\bullet$ is flat over $S$, the restrictions $\cM_\bullet \otimes K$ and $\cM_\bullet \otimes k$ are quasi-isomorphic to $\cL_\bullet \otimes K$ and $\cL_\bullet \otimes k$, respectively \cite[Chapter~0, Proposition~11.9.1; see also Remark~11.9.3]{EGA3.1}.  We claim that the homology of $\cM_\bullet$ vanishes in high degree: indeed, the homology of $\cM_{\bullet} \otimes k$ computes $\Tor$  on the smooth $\cY_k$, so vanishes in high degree, and the claim then follows from flatness of $\cM_{\bullet}$ and Nakayama's lemma.  Thus, we can truncate the complex, replacing some $\cM_i$ by the image of $\cM_{i+1}$, which is coherent and torsion-free and hence free of finite rank.  In particular, we may assume that $\cM_\bullet$ is also bounded above.  It is then clear that
\[
i(\cX_s \cdot \cX'_s, \cY_s) = \sum_i (-1)^i \rk \cM_i,
\]
for $s$ equal to either the generic or closed point of $S$, and the theorem follows.
\end{proof}

\begin{rem} \label{rem:noetherian-mults}
Theorem~\ref{thm:int-mults} can also be proved over an arbitrary base scheme as follows.  It is easy to see that intersection numbers are invariant under extension of the base field, so the result behaves well under base change.  Passing to the closure of the more general point, and then taking an affine neighborhood around the special point, we reduce to the case where the base is affine and irreducible.  This ensures that $X$ and $X'$ are finitely presented over the base $S$, because they are flat and finite type, and then we can proceed by noetherian approximation.  Say the affine base is $\Spec A$.  Then there is a finitely generated subalgebra $A_0$ of $A$ over which $X$, $X'$, and $Y$ and all of the relevant morphisms are defined, and all can be chosen so that the relevant geometric properties are preserved, including that the models of $X$ and $X'$ are flat \cite[Sections 11 and 12]{EGA4.3}, and the model of $Y$ is quasiprojective \cite[Theorem 8.10.5]{EGA4.3} and smooth \cite[Section~17.7]{EGA4.4}.  In particular, we may assume the base is noetherian.  By \cite[Proposition 7.1.9]{EGA2} we can thus reduce to the case of a DVR, and the theorem follows from the well-known fact that intersection numbers are preserved by specialization over a DVR \cite[Section~20.3]{IT}.
\end{rem}

In the geometric case that we are interested in, where the fraction field $K$ is algebraically closed and hence the residue field $k$ is algebraically closed as well, we apply Theorem~\ref{thm:int-mults} on intersection numbers to get the following result on individual intersection multiplicities along components of the intersection of the special fibers.  Note that this result is of a local nature, and holds for arbitrary expected dimension for the intersection.

\begin{thm}\label{thm:int-mults-local} 
Assume $K$ is algebraically closed.  Let $\cY$ be smooth over $S$, and let $\cX$ and $\cX'$ be closed subschemes that are flat over $S$.  Suppose $\cX_k$ and $\cX'_k$ meet properly along a component $Z$ of their intersection.
Then 
\[
i(Z,\cX_k \cdot \cX'_k;\cY_k)=\sum_{\widetilde{Z}} m(Z,\widetilde{Z}) \cdot 
i(\widetilde{Z},\cX_K \cdot \cX'_K; \cY_K),
\] 
where the sum is over irreducible components $\widetilde{Z}$ of $\cX_K \cap \cX'_K$
whose closures contain $Z$, and $m(Z,\widetilde{Z})$ is the multiplicity
of $Z$ in the special fiber of the closure of $\widetilde{Z}$ inside $\cY$.
\end{thm}

\begin{proof}
We first prove the special case where $\cX$ and $\cX'$ have complementary dimension.  Suppose $Z$ is a point.  We claim that there is an affine open neighborhood $\cU$ of $Z$ such that $\cX \cap \cX' \cap \cU$ is finite over $S$ and has special fiber exactly $Z$.  By upper semicontinuity of fiber dimension \cite[Theorem~13.1.3]{EGA4.3}, the union $\cW_1$ of the positive dimensional components of $\cX_K \cap \cX'_K$ and $\cX_k \cap \cX'_k$ is closed in $\cY$.  The intersection of $\cX_K$ and $\cX'_K$ in the complement of $\cW_1$ is a finite set of $K$-points.  Let $\cW_2$ be the closure of those points in this set that do not specialize to $Z$.  We claim that any affine neighborhood of $Z$ contained in $\cY \smallsetminus( \cW_1 \cup \cW_2)$ is the desired neighborhood.  Now $\cX \cap \cX' \cap \cU$ is separated, quasifinite over $S$ and has special fiber $Z$, by construction, and it is locally of finite presentation, since $\cX$ and  $\cX'$ are locally of finite type and flat, and hence locally of finite presentation.
By \cite[Theorem~8.11.1]{EGA4.3}, to show that this intersection is finite it only remains to check that it is proper over $S$.  The generic fiber consists of finitely many $K$-points, each of which extends to an $R$-point, and one can check that this implies the valuative criterion for universal closedness \cite[Theorem~7.3.8]{EGA2}, which proves the claim.  Now, note that $m(\widetilde Z, Z)$ is 1 for every $\widetilde Z$, since each $\widetilde Z$ is a $K$-point whose closure is a section.  The required equality of intersection numbers then follows immediately from Theorem~\ref{thm:int-mults}.
 
We now prove the general case, where the intersection may have positive dimension, by reducing to the case of a zero-dimensional intersection.  Since the statement is local on $\cY$, we may assume $\cY$ is affine, that $Z$ is the only component of $\cX_k \cap \cX'_k$, and that every component of $\cX_K \cap \cX'_K$ contains $Z$ in its closure.  Let $\widetilde{Z}_1,\dots,\widetilde{Z}_m$ be 
the irreducible components of $\cX_K \cap \cX'_K$. Localizing further, we may also assume that $Z$ and the $\widetilde{Z}_i$ are all smooth (although we cannot assume that the closures of the $\widetilde{Z}_i$ are smooth, as $Z$ may appear with multiplicity greater than $1$). 

Now, let $\cL_k$ be a general linear subspace of complementary dimension in the special fiber, so $\cL_k$ meets $Z$ transversely at finitely many points.  Choose a general lift $\cL$ of $\cL_k$ to $S$; by the Zariski density of Theorem \ref{thm:closed-pts} applied to the Grassmannian over $S$, the generic fiber $\cL_K$ meets each component of $\cX_K \cap \cX'_K$ transversely, at isolated points.  After further localizing, we may assume $\cL_k$ meets $Z$ at a single point $z$, and every point $\widetilde z$ of $\cX_K \cap \cX'_K \cap \cL_K$ specializes to $z$.  Since $\cL$ is a complete intersection, applying \cite[Example~8.1.10]{IT} in the special fiber, we have
\[
i(Z,\cX_k \cdot \cX'_k;\cY_k)= i(z, (\cX_k \cap \cL_k) \cdot (\cX'_k \cap \cL_k); \cL_k).
\]
Now, note that an inductive application of \cite[Proposition~11.3.7]{EGA4.3} implies that $\cX \cap \cL$ and $\cX' \cap \cL$ remain flat over $S$ so,
by the zero-dimensional case treated above, the right hand side is equal to
\[
\sum_{\widetilde z} i(\widetilde z, (\cX_K \cap \cL_K) \cdot (\cX'_K \cap \cL_K); \cL_K).
\]
If $\widetilde Z_i$ is the component of $\cX_K \cap \cX'_K$ containing $\widetilde z$, then a similar argument in the generic fiber gives
\[
i(\widetilde z, (\cX_K \cap \cL_K) \cdot (\cX'_K \cap \cL_K); \cL_K) = i(\widetilde Z_i,\cX_K \cdot \cX'_K;\cY_K).
\]
The theorem follows since, by the zero-dimensional case above, the multiplicity $m(\widetilde Z_i, Z)$ is the number of points $\widetilde z$ in $\widetilde Z_i \cap \cL_K$ specializing to $z$.
\end{proof}

\begin{cor} \label{cor:cycles}
Let $X$ be a pure-dimensional closed subscheme of $T$. Then $\Trop(X)$ is equal to the tropicalization of the fundamental cycle $\Trop([X])$.
\end{cor}

\begin{proof}
We may assume the valuation is nontrivial.  It is clear that the tropicalizations agree set-theoretically, since both are the closure of the image of $X(K)$.  To see that the multiplicities agree, let $w$ be a point in the relative interior of a facet of $\Trop(X)$, and apply Theorem~\ref{thm:int-mults-local} in the special case where $\cY$ and $\cX'$ are both equal to $\cT^w$.
\end{proof}

\section{Lifting tropical intersections with multiplicities} \label{sec:main}

We now use the results of Sections~\ref{sec:trivial val} and \ref{sec:dim} to prove the main tropical lifting theorems, both as stated in the introduction and in refined forms with multiplicities.  We also prove generalizations to intersections of three or more subschemes and discuss weaker lifting results when the tropicalizations do not meet properly.

\subsection{Intersections of two subschemes}

Let $Y$ be a closed subvariety of $T$, and let $X$ and $X'$ be closed subschemes of pure codimension $j$ and $j'$, respectively, in $Y$.

\begin{proof}[Proof of Theorem~\ref{thm:lift-to-degens}]
Suppose $\Trop(X)$ meets $\Trop(X')$ properly at a simple point $w$ of $\Trop(Y)$.  Let $\widetilde K | K$ be an extension of valued fields such that $w \in N_{\widetilde G}$, where $\widetilde G$ is the value group of $G$.  Then the tropicalizations after base change $\Trop(\widetilde X)$ and $\Trop(\widetilde X')$ meet properly at $w$, and Theorem~\ref{thm:initial base change} implies furthermore that $X_w$ meets $X'_w$ properly at a smooth point of $Y_w$ if and only if $\widetilde X_w$ meets $\widetilde X'_w$ properly at a smooth point of $\widetilde Y_w$.  Therefore, after an extension of valued fields, we may assume $w \in N_G$. 

The initial degeneration $Y_w$ is a torus torsor containing $X_w$ and $X'_w$, and the tropicalizations $\Trop(X_w)$ and $\Trop(X'_w)$ with respect to the trivial valuation are the stars of $w$ in $\Trop(X)$ and $\Trop(X')$, respectively, and hence meet properly in the vector space $\Trop(Y_w)$.  Theorem~\ref{thm:trivial val} then says that 
\[
\Trop(X_w \cap X'_w) = \Trop(X_w) \cap \Trop(X'_w).
\]
In particular, $X_w \cap X'_w$ is nonempty of codimension $j + j'$, and the theorem follows, since $Y_w$ is smooth.
\end{proof}

\begin{proof}[Proof of Theorem \ref{exploded main}] 
Suppose $X_w$ meets $X'_w$ properly at a smooth point $x$ of $Y_w$.  After extending the valued field we may assume the valuation is nontrivial and $w \in N_G$.  In this case, the integral model $\cY_w$ is of finite type over $\Spec R$, by Proposition~\ref{prop:finite type}.  We can then pass to an open neighborhood of $x$ in $\cY_w$ in which the special fiber is smooth and apply Theorem~\ref{thm:lift to K} to deduce that there is a point of $X \cap X'$ specializing to $x$.  Therefore $x$ is contained in $(X \cap X')_w$.
\end{proof}

As noted in the introduction, Theorem \ref{main} follows immediately from Theorems~\ref{thm:lift-to-degens} and \ref{thm:lift to K}, and Theorem~\ref{basic} is the special case where $Y$ is the torus $T$.  So this concludes the proof of the theorems stated in the introduction.  

\smallskip

We now state and prove a refined version of Theorem~\ref{main} with multiplicities.  If $\Trop(X)$ and $\Trop(X')$ meet properly at a simple point $w$ of $\Trop(Y)$, and $\tau$ is a facet of $\Trop(X) \cap \Trop(X')$ containing $w$, then the tropical intersection multiplicity $i(\tau, \Trop(X) \cdot \Trop(X'); \Trop(Y))$  along $\tau$ in $\Trop(Y)$ is defined by a local displacement rule, as in Section~\ref{sec:stable}, but the displacement vector is required to be in the subspace of $N_\RR$ parallel to the affine span of the facet of $\Trop(Y)$ containing $w$.  Then $i(\tau, \Trop(X) \cdot \Trop(X'); \Trop(Y))$ is the multiplicity of $\tau_w$ in the stable tropical intersection of the stars of $w$ in $\Trop(X)$ and $\Trop(X')$, inside the star of $w$ in $\Trop(Y)$.

\begin{thm}\label{thm:main plus multiplicities}
Suppose $\Trop(X)$ meets $\Trop(X')$ properly along a face $\tau$ of codimension $j + j'$ that contains a simple point of $\Trop(Y)$.
Then $\tau$ is contained in $\Trop(X \cap X')$ with multiplicity bounded below by the tropical intersection multiplicity
\[
m_{X \cap X'}(\tau) \geq i(\tau,\Trop(X) \cdot \Trop(X');\Trop(Y)),
\]
and both are strictly positive.  Furthermore, the tropical intersection multiplicity is equal to the weighted sum of algebraic intersection multiplicities
\[
i(\tau,\Trop(X)\cdot \Trop(X');\Trop(Y)) = \sum_{Z} i(Z,X \cdot X';Y) \, m_Z(\tau),
\]
where the sum is over components $Z$ of $X \cap X'$ whose tropicalizations contain $\tau$.
\end{thm}

Although the global intersection product $X \cdot X'$ need not be defined, the intersection multiplicities $i (Z, X \cdot X';Y)$ (as defined in terms of $\Tor$ in Section~\ref{sec:toric-intersect}) appearing in the statement of Theorems~\ref{thm:main plus multiplicities} and Theorem~\ref{thm:main-multiple}, below, are nonetheless well-defined.  Indeed, for every component $Z$ of $X \cap X'$ whose tropicalization contains $\tau$, we have that $Y$ must be smooth on a non-empty open subset that intersects $Z$ by the simple point hypothesis, and because $\Trop(X)$ is assumed to meet $\Trop(X')$ properly along $\tau$, it also follows that $X$ meets $X'$ properly along $Z$. 

\begin{proof} 
First, by Theorem \ref{main}, the face $\tau$ is contained in $\Trop(X \cap X')$.  We prove the identity between the tropical intersection multiplicity and the weighted sum of algebraic intersection multiplicities and then deduce the inequality for $m_{X \cap X'}(\tau)$, as follows.  Let $w$ be a point in $\tau$ that is a simple point of $\Trop(Y)$.  After extending the valued field, we may assume $w \in N_G$.  Then the tropical intersection multiplicity along $\tau$ is equal to the local tropical intersection multiplicity at $w$
\[
i(\tau, \Trop(X) \cdot \Trop(X'); \Trop(Y)) = i (\tau_w, \Trop(X_w) \cdot \Trop(X'_w); \Trop(Y_w)),
\]
where $\tau_w = \RR_{\geq 0} (\tau - w)$, as in Section~\ref{sec:compatibility}.
Then, since $Y_w$ is a torus torsor, Theorem~\ref{thm:trivial val} says that this local tropical intersection multiplicity is given by
\[
 i (\tau_w, \Trop(X_w) \cdot \Trop(X'_w); \Trop(Y_w)) \, = \, \sum_{Z_0} i(Z_0, X_w \cdot X'_w; Y_w) m_{Z_0}(\tau_w),
\]
where the sum is over components $Z_0$ of $X_w \cap X'_w$ whose tropicalizations contain $\tau_w$.  By Theorem~\ref{thm:int-mults-local}, for each such $Z_0$,
\[
i(Z_0, X_w \cdot X'_w; Y_w) = \sum_Z m(Z, Z_0) \, i(Z, X \cdot X'; Y),
\]
where the sum is over components $Z$ of $X \cap X'$ and $m(Z,Z_0)$ is the multiplicity of $Z_0$ in the special fiber of the closure of $Z$ in $\cY_w$.  Combining these identities with the results of Section~\ref{sec:compatibility} gives the required equality, since $m_Z(\tau) = \sum_{Z_0} m(Z, Z_0) m_{Z_0}(\tau_w)$, where the sum ranges over components $Z_0$ of the special fiber of the closure of $Z$.

The inequality follows immediately from this equality, since
\[
m_{X \cap X'}(\tau) = \sum_Z \length( \cO_{Z, X \cap X'}) m_Z (\tau),
\]
by Corollary~\ref{cor:cycles}, and $\length( \cO_{Z, X \cap X'}) \geq i(Z, X \cdot X'; Y)$ for each $Z$ \cite[Proposition~7.1]{IT}.
\end{proof}

If $\Trop(X)$ meets $\Trop(X')$ properly in $N_\RR$, then part (2) of Theorem~\ref{thm:main plus multiplicities} has a particularly 
simple statement equating the tropicalization of the refined intersection cycle with the stable tropical intersection.

\begin{cor}\label{cor:main-cycles} Suppose $\Trop(X)$ meets $\Trop(X')$ properly in $N_\RR$.  Then 
$$\Trop(X \cdot X')=\Trop(X) \cdot \Trop(X').$$
\end{cor}

We also note that the inequality in Theorem~\ref{thm:main plus multiplicities} can be replaced by an equality if $X$ and $X'$ are smooth, or mildly singular.

\begin{cor} \label{cor:CM}
Under the hypotheses of Theorem~\ref{thm:main plus multiplicities}, if $X$ and $X'$ are Cohen-Macaulay then
\[
m_{X \cap X'}(\tau) = i(\tau, \Trop(X) \cdot \Trop(X'); \Trop(Y)).
\]
\end{cor}

\begin{proof}
If $X$ and $X'$ are Cohen-Macaulay, then the length of the scheme-theoretic intersection of $X$ and $X'$ along each component $Z$ of expected dimension is equal to the intersection multiplicity $i(Z, X \cdot X'; Y)$ \cite[Example~8.2.7]{IT}.  The result then follows from the equality in Theorem~\ref{thm:main plus multiplicities}, since
\[
m_{X \cap X'}(\tau) = \sum_Z \length( \cO_{Z, X \cap X'}) m_Z(\tau),
\] 
by Lemma~\ref{lem:linearity}, where the sum is over components $Z$ of $X \cap X'$ such that $\Trop(Z)$ contains $\tau$.
\end{proof}

\subsection{Intersections of three or more subschemes}

In applications, and particularly in the context of enumerative geometry, one frequently wants to intersect more than two subschemes.  The algebraic intersection product of several closed subschemes $X_1 \cdots X_r$ may be treated either by induction from results on intersection of pairs or by a standard reduction to the diagonal argument.  For the reduction to the diagonal, one uses the facts that $X_1 \cap \cdots \cap X_r$ is canonically identified with the intersection $X_1 \times \cdots \times X_r \cap \Delta$ in $Y^r$, where $\Delta$ is the diagonal subscheme, and that $X_1 \times \cdots \times X_r \cdot \Delta$ is the push forward of $X_1 \cdots X_r$ under the diagonal embedding.  

\begin{defn}
Let $X_1, \ldots, X_r$ be closed subschemes in $T$, of pure codimension $j_1, \ldots, j_r$, respectively.  Then the stable tropical intersection $\Trop(X_1) \cdots \Trop(X_r)$ is the iterated pairwise stable tropical intersection
\[
\Trop(X_1) \cdots \Trop(X_r) = (\cdots ((\Trop(X_1) \cdot \Trop(X_2)) \cdot \Trop(X_3)) \cdots \Trop(X_r)).
\]
\end{defn}
\noindent The following proposition shows that stable tropical intersections for three or more closed subschemes can be computed as a pairwise intersection with the tropicalization of the diagonal.  In particular, the stable tropical intersection is independent of the order of the factors.

\begin{prop} \label{prop:stable-multiple}
The stable tropical intersection $\Trop(X_1 \times \cdots \times X_r) \cdot \Trop(\Delta)$ is the image of $\Trop(X_1) \cdots \Trop(X_r)$ under the diagonal embedding of $N_\RR$ in $N_\RR^r$.
\end{prop}

\begin{proof}
Let $\Sigma$ be a complete unimodular fan that contains $\Trop(X_1), \ldots, \Trop(X_r)$ as subfans.  By Proposition~\ref{prop:local int}, the stable tropical intersections $\Trop(X_1 \times \cdots \times X_r) \cdot \Trop(\Delta)$ and $\Trop(X_1) \cdots \Trop(X_r)$ correspond to the products of Minkowski weights $c_{\overline{X_1 \times \cdots \times X_r}} \cdot c_{\overline \Delta}$ and $c_{\overline X_1} \cdots c_{\overline X_r}$, on $\Sigma^r$ and $\Sigma$, respectively.  These products of Minkowski weights are equal to $c_{\overline X_1 \cdots \overline X_r}$ and $c_{\overline{X_1 \times \cdots \times X_r} \cdot \overline \Delta}$, by the identification of rings of Minkowski weights with Chow rings, discussed in Section~\ref{sec:toric-intersect}.  The proposition follows, since $\overline{X_1 \times \cdots \times X_r} \cdot \overline \Delta$ is the push forward of $\overline X_1 \cdots \overline X_r$ under the diagonal embedding.
\end{proof}

As in the case of pairwise intersections, when $X_1, \ldots, X_r$ are subschemes of pure codimension $j_1, \ldots, j_r$ in an irreducible variety $Y$, we define the stable tropical intersection multiplicity along a face $\tau$ of codimension $j_1 + \cdots + j_r$ in $\Trop(Y)$, provided that $\tau$ contains a simple point $w$ of $\Trop(Y)$, as the multiplicity of $\tau_w$ in the stable tropical intersection of the stars of $w$ in $\Trop(X_1), \ldots, \Trop(X_r)$ inside the star of $w$ in $\Trop(Y)$.

We now generalize Theorem~\ref{thm:main plus multiplicities} to intersections of several closed subschemes by reducing to the case of a pairwise intersection with the diagonal in $Y^r$.

\begin{thm}\label{thm:main-multiple} 
Let $X_1,\dots,X_r$ be closed subschemes of pure codimension $j_1,\dots,j_r$ in $Y$, respectively.  Suppose $\tau$ is a facet of 
$\Trop(X_1) \cap \cdots \cap \Trop(X_r)$ of codimension $j_1 + \cdots + j_r$ in $\Trop(Y)$ that contains
a simple point of $\Trop(Y)$.  Then the tropicalization $\Trop(X_1 \cap \cdots \cap X_r)$ contains $\tau$ with multiplicity bounded below by the tropical intersection multiplicity
\[
m_{X_1 \cap \cdots \cap X_r}(\tau) \geq i(\tau,\Trop(X_1) \cdots \Trop(X_r);\Trop(Y)),
\]
and both are strictly positive.  Furthermore, the tropical intersection multiplicity is equal to the weighted sum of algebraic intersection multiplicities
\[
i(\tau,\Trop(X_1)\cdots \Trop(X_r);\Trop(Y))= \sum_{Z} i(Z,X_1 \cdots X_r;Y) m_Z(\tau),
\]
where the sum is over components of $X_1 \cap \cdots \cap X_r$ such that $\Trop(Z)$ contains $\tau$.
\end{thm}

\begin{proof}
Let $\tau$ be a facet of $\Trop(X_1) \cap \cdots \cap \Trop(X_r)$ of codimension $j_1 + \cdots + j_r$ in $\Trop(Y)$ that contains a simple point $w$ of $\Trop(Y)$.  Then the image of $\tau$ under the diagonal embedding $\delta$ of $\Trop(Y)$ in $\Trop(Y)^r$ is a facet of $\Trop(X_1 \times \cdots \times X_r) \cap \Trop(\Delta)$ of codimension $(r-1) \dim(Y) + j_1 + \cdots + j_r$ in $\Trop(Y^r)$, where $\Delta$ is the diagonal subscheme in $Y^r$, containing a simple point of $\Trop(Y^r)$.  By Theorem~\ref{thm:main plus multiplicities}, $\delta(\tau)$ is a face of $\Trop(X_1 \times \cdots \times X_r \cap \Delta)$.  Furthermore, the multiplicity of $\delta(\tau)$ in the stable tropical intersection $\Trop(X_1 \times \cdots \times X_r) \cdot \Trop(\Delta)$ inside $\Trop(Y^r)$ is strictly positive, equal to
\[
\sum_Z i(Z, X_1 \times \cdots \times X_r \cdot \Delta; Y^r) m_Z(\delta(\tau)),
\] 
and less than or equal to $m_{X_1 \times \cdots \times X_r \cap \Delta}(\delta(\tau))$.  Now $X_1 \times \cdots \times X_r \cap \Delta$ is canonically identified with $X_1 \cap \cdots \cap X_r$, and both the tropical multiplicities and the local intersection multiplicities agree.  In other words, we have
\[
i(Z, X_1 \times \cdots \times X_r \cdot \Delta; Y^r) = i(Z, X_1 \cdots X_r; Y)
\]
and
\[
m_Z(\delta(\tau)) = m_Z(\tau).
\]
The theorem then follows, by Proposition~\ref{prop:stable-multiple}.
\end{proof}


The following application of Theorem~\ref{thm:main-multiple}, giving a formula for counting points in zero-dimensional complete intersections, has been applied by Rabinoff to construct canonical subgroups of abelian varieties over $p$-adic fields, via a suitable generalization for power series \cite{Rabinoff12b, Rabinoff12}.  

Let $f_i = \sum a_i(u) x^u$ be an equation defining $X_i$, with $P_i = \conv \{u | a_i \mbox{ is nonzero} \}$ its Newton polytope.  Projecting the lower faces of the lifted Newton polytope $\conv \{ (u, \nu(a_i(u)) \}$ in $M_\RR \times \RR$ gives the \textbf{Newton subdivision} of $P_i$.  This Newton subdivision is dual to $\Trop(X_i)$, in the sense that there is a natural polyhedral structure on $\Trop(X_i)$ whose faces are in order-reversing bijection with the positive dimensional faces of the Newton subdivision.  A face $\tau$ of $\Trop(X_i)$ corresponds to the convex hull of the lattice points $u$ such that $a_i(u)x^u$ is a monomial of minimal $w$-weight, for $w$ in the relative interior of $\tau$.

Our formula is phrased in terms of mixed volumes of faces of the Newton subdivision.  Recall that, for lattice polytopes $Q_1, \ldots, Q_n$ in $M_\RR \cong \RR^n$, the euclidean volume of the Minkowski sum $b_1 Q_1 + \cdots + b_n Q_n$ is a polynomial of degree $n$ in $b_1, \ldots, b_n$, and the mixed volume $V(Q_1, \ldots, Q_n)$ is the coefficient of $b_1 \cdots b_n$ divided by $n!$.  If $\Sigma$ is the smallest common refinement of the inner normal fans of the $Q_i$, then each $Q_i$ corresponds to a nef line bundle $L_i$ on the toric variety $Y(\Sigma)$, and the mixed volume is equal to the intersection number, divided by $n!$,
\[
V(Q_1, \ldots, Q_n) = (c_1(L_1) \cdots c_1(L_n))/n!.
\]
See \cite[p.~116]{Fulton93} for further details on mixed volumes and their relation to toric intersection theory\footnote{There is a minor misstatement in the definition of mixed volumes, in the text in between displayed formulas (1) and (2) of \cite[p.~116]{Fulton93}.  The mixed volume is the coefficient of $\nu_1 \cdots \nu_n$ divided by $n!$, not multiplied by $n!$.  The displayed formulas (1), (2), and (3) are correct.  Formulas (1) and (2) uniquely determine the mixed volumes of rational polytopes, as does (3), which also characterizes mixed volumes of arbitrary convex bodies.} and \cite[Section~4]{KavehKhovanskii12} for generalizations to arbitrary projective varieties via Newton-Okounkov bodies.  Mixed volume formulas for tropical stable complete intersections are standard in the case of the trivial valuation.  The earliest reference we know of relating mixed volumes to tropical complete intersections for a nontrivial (discrete) valuation is due to Smirnov \cite{Smirnov97}.

\begin{cor} \label{cor:complete-intersection}
Let $X_1, \ldots, X_n$ be hypersurfaces in $T$, and suppose $w$ is an isolated point in $\Trop(X_1) \cap \cdots \cap \Trop(X_n)$.  Let $Q_i$ be the face of the Newton sub\-di\-vision corresponding to the minimal face of $\Trop(X_i)$ that contains $w$.  Then the number of points in $X_1 \cap \cdots \cap X_n$ with tropicalization $w$, counted with multiplicities, is exactly $n! V (Q_1, \ldots, Q_n)$.
\end{cor}

\begin{proof}
Fix a complete unimodular fan $\Sigma$ in $N_\RR$ such that, for $i = 1, ..., n$,  every face of $\Star_w \Trop(X_i)$ is a union of faces of $\Sigma$.  Since $\Star_w \Trop(X_i)$ is the codimension 1 skeleton of the possibly degenerate inner normal fan of $Q_i$, it follows that $Q_i$ corresponds to a nef line bundle $L_i$ on $Y(\Sigma)$.  Furthermore, the Minkowski weight of codimension 1 on $\Sigma$ given by the tropical multiplicities on $\Star_w \Trop(X_i)$ corresponds to $c_1(L_i)$.  It follows, by the compatibility of toric and stable tropical intersections that the tropical intersection multiplicity
\[
i(w, \Trop(X_1) \cdots \Trop(X_n)) = (c_1(L_1) \cdots c_1(L_n)),
\]
and the latter is $n! V(Q_1, \ldots, Q_n)$.  By Theorem~\ref{thm:main-multiple}, this tropical intersection multiplicity is equal to the number of points in $X_1 \cap \cdots \cap X_n$ with tropicalization $w$, counted with multiplicities, as required.
\end{proof}

\begin{rem}
In Corollary~\ref{cor:complete-intersection}, the multiplicities on the points $x$ in $X_1 \cap \cdots \cap X_n$ with tropicalization $w$ are equal to the lengths of the scheme-theoretic intersection of $X_1, \ldots, X_n$ along $x$, as in Corollary~\ref{cor:CM}, since complete intersections are Cohen-Macaulay.
\end{rem}

\subsection{Non-proper intersections}

Our main tropical lifting results require the tropicalizations to meet properly.  Nevertheless, our results on lifting from the special fiber to the generic fiber of $\cY^w$, such as Theorem~\ref{exploded main}, still yield nontrivial statements for nonproper tropical intersections.  In many cases, such as Example~\ref{ex:lines}, the initial degenerations meet properly even when the tropicalizations do not.

\begin{prop}\label{cor:degens-proper} 
Suppose that for each $w \in \Trop(X) \cap \Trop(X')$, we have that $Y_w$ is 
smooth and $X_w$ meets $X'_w$ properly in $Y_w$.  Then $X$ meets $X'$ properly in $Y$.  Furthermore,  the set of $w$ such that $X_w \cap X'_w$ is nonempty is either empty or the underlying set of a polyhedral complex of pure codimension $j+j'$ in $\Trop(Y)$. 
\end{prop}

\begin{proof}
If $w$ is in $\Trop(X \cap X')$, then $X_w$ meets $X'_w$ along $(X \cap X')_w$.  
Conversely, if $X_w$ meets $X'_w$ properly at some smooth point $x$ of $Y_w$ then $x$ is contained in $(X \cap X')_w$, by Theorem~\ref{exploded main}, and hence $w$ is in $\Trop(X \cap X')$.  Therefore, the hypotheses of the proposition ensure that the set of $w$ such that $X_w$ meets $X'_w$ is exactly $\Trop(X \cap X')$, and $(X \cap X')_w$ is equal to $X_w \cap X'_w$, for all $w$.  If $X_w \cap X'_w$ is nonempty then it has pure codimension $j + j'$ in $Y_w$.  It follows that $X \cap X'$ has pure codimension $j + j'$ in $Y$, and hence $\Trop(X \cap X')$ is a polyhedral complex of pure codimension $j + j'$ in $\Trop(Y)$, as required. 
\end{proof}

In the special case where $Y$ is the ambient torus $T$, one can say even more.  If all initial degenerations meet properly, then the set of weight vectors where the intersection is nonempty is exactly the stable tropical intersection, as was suggested to us by J. Rau.

\begin{prop}
Suppose $X_w$ meets $X'_w$ properly in $T_w$ for all $w$.  Then the set of $w \in N_\RR$ such that $X_w \cap X'_w$ is nonempty is exactly the underlying set of the stable tropical intersection $\Trop(X) \cdot \Trop(X')$.
\end{prop}

\begin{proof}
Let $\Sigma$ be a complete unimodular fan that contains the stars of $w$ in $\Trop(X)$ and in $\Trop(X')$ as subfans.  First we claim that, for any $w$ in $N_\RR$, the closures $\overline X_w$ and $\overline X'_w$ meet properly in $Y(\Sigma)$ and, furthermore, $\overline X_w \cap O_\sigma$ and $\overline X'_w \cap O_\sigma$ meet properly in $O_\sigma$ for every $\sigma$ in $\Sigma$.  Indeed, after choosing an extension of valued fields such that $w$ is rational over the value group and subdividing $\Sigma$, we may assume that it contains tropical fans, in the sense of \cite{Tevelev07}, for $X_w$ and $X'_w$ as subfans.  Then, for $v \in N_G$ in the relative interior of $\sigma$ and $\epsilon \in G$ sufficiently small and positive, the initial degenerations $X_{w + \epsilon v}$ and $X'_{w + \epsilon v}$ agree with the initial degenerations of $X_w$ and $X'_w$, respectively, for the weight vector $v$ \cite[Corollary~10.12]{Gubler12}.  The images of these initial degenerations under projection to $O_\sigma$ are identified with $\overline X_w \cap O_\sigma$ and $\overline X'_w \cap O_\sigma$, up to simultaneous translation in $O_\sigma$ \cite[Remark~12.7]{Gubler12}.  Since $X_{w + \epsilon v}$ and $X'_{w + \epsilon v}$ meet properly in $T_{w + \epsilon v}$, it follows that $\overline X_w \cap O_\sigma$ and $\overline X'_w \cap O_\sigma$ meet properly in $O_\sigma$, as claimed.

Since $\overline X_w \cap O_\sigma$ and $\overline X'_w \cap O_\sigma$ meet properly in $O_\sigma$ for every $\sigma$ in $\Sigma$, Proposition~\ref{prop:cycles} says that a cone $\tau_w$ of codimension $j + j'$ in $\Sigma$ appears in the tropicalization of $X_w \cdot X'_w$ with multiplicity equal to the stable tropical intersection multiplicity $i(\tau_w, \Trop(X_w) \cdot \Trop(X'_w))$ which agrees with the stable tropical intersection multiplicity $i(\tau, \Trop(X) \cdot \Trop(X')$, by Proposition~\ref{prop:local int}.  In particular, $X_w \cap X'_w$ is nonempty if and only if $w$ is contained in the stable tropical intersection $\Trop(X) \cdot \Trop(X')$.
\end{proof}

\smallskip
For closed subschemes of the torus $T$, when a tropical intersection is not proper we can translate the subschemes by a suitable element of $T(K)$ so that the initial degenerations meet properly, without changing the tropicalizations.  The following theorem extends Proposition~\ref{prop:translation} to the general case, where the valuation may not be trivial, and gives a geometric meaning to the stable tropical intersection, as the tropicalization of an intersection with a translate by a general point $t$ such that $\Trop(t)$ is zero.

If $t$ is in $T(K)$ then the tropicalization of the translate $tX$ is the translate of $\Trop(X)$ by the vector $\Trop(t)$ in $N_G$.  In particular, the tropicalization is invariant under translation by points $t$ such that $\Trop(t)$ is zero.  The set of such points is Zariski dense in $T(K)$ \cite[Corollary~4.2]{tropicalfibers} (see also \cite[Remark~2]{TropicalFibers-Correction}), but not Zariski open if the valuation is nontrivial.  Nevertheless, we say that a property holds for a general point $t$ such that $\Trop(t)$ is zero if there is a Zariski open subset $U_0$ of the initial degeneration $T_0$ such that the property holds for all $t$ with $t_0 \in U_0$.  This set is again Zariski dense in $T(K)$.

\begin{thm} \label{thm:translation}
Let $X$ and $X'$ be pure-dimensional closed subschemes of $T$.   Then, for a general point $t$ such that $\Trop(t)$ is zero,
\[
\Trop(X \cap tX') = \Trop(X) \cdot \Trop(X').
\]
\end{thm}

\begin{proof}
If the valuation is trivial then the theorem is given by Proposition~\ref{prop:translation}.  Assume the valuation is nontrivial.  Let $\tau$ be a face of $\Trop(X) \cap \Trop(X')$, and let $w \in N_G$ be a point in $\tau$.  The initial degeneration $(tX')_w$ is the translate of $X'_w$ by $t_0$.  Since $t_0$ is general in $T_0$, Proposition~\ref{prop:translation} says that
\[
\Trop(X_w \cap (tX')_w) = \Trop(X_w) \cdot \Trop(X'_w).
\]
In particular, if $w$ is not in the stable tropical intersection then the initial degenerations $X_w$ and $(tX')_w$ are disjoint, so $w$ is not in $\Trop(X \cap tX')$.  On the other hand, if $w$ is in the stable intersection then $X_w$ meets $(tX')_w$ properly in $T_w$, by Proposition~\ref{prop:translation}, and all points of $X_w \cap (tX')_w$ lift to the generic fiber $X \cap tX'$, by Theorem~\ref{thm:lift to K}.  Note that for $w$ and $w'$ in the relative interior of $\tau$, because the isomorphisms $X_w \cong X_{w'}$ and $tX'_w \cong tX'_{w'}$ can be simultaneously induced by a single isomorphism $T_w \cong T_{w'}$, if $X_w$ meets $tX'_w$ then $X_{w'}$ likewise meets $tX'_{w'}$ properly.   Since there are only finitely many faces of $\Trop(X) \cap \Trop(X')$, the point $t$ can be chosen sufficiently general so that this holds in every face.  

This shows that the underlying sets of $\Trop(X \cap tX')$ and $\Trop(X) \cdot \Trop(X')$ are equal.  Now, let $\tau$ be a facet of $\Trop(X) \cdot \Trop(X')$ and let $w \in N_G$ be a point in the relative interior of $\tau$.  The tropical intersection multiplicity along $\tau$ is equal to the local tropical intersection multiplicity at $w$
\[
i(\tau, \Trop(X) \cdot \Trop(X'); \Trop(T)) = i (\tau_w, \Trop(X_w) \cdot \Trop(X'_w); \Trop(T_w)),
\]
and, by Proposition~\ref{prop:translation}, the right hand side is equal to $m_{X_w \cap tX'_w} (\tau_w)$.  Applying Theorem~\ref{thm:int-mults-local} to the intersection of $\cX^w$ and $(t\cX')^w$ in $\cT^w$, as in the proof of Theorem~\ref{thm:main plus multiplicities}, then shows that $m_{X_w \cap tX'_w} (\tau_w)$ is equal to $m_{X \cdot tX'}(\tau)$.  As in the proof of Proposition~\ref{prop:translation}, the genericity of $t$ guarantees that the cycle $X \cdot tX'$ is equal to the fundamental cycle of $X \cap tX'$.  Therefore, $m_{X \cdot tX'}(\tau)$ is equal to $m_{X \cap tX'}(\tau)$, and the theorem follows.
\end{proof}

\section{Examples} \label{sec:examples}

Here we give a number of examples illustrating our tropical lifting theorems and the necessity of their hypotheses.  The first example involves tropicalizations that meet properly, but not necessarily in the interiors of maximal faces.

\begin{ex2}\label{ex:parabola} 
Inside $Y = (K^*)^2$, let $X$ and $X'$ be given by $y= x + 1$ and $y = ax^2$, respectively, with $a \in K^*$. We consider three cases, according to whether or not $\nu(a)$ is zero, and its sign if it is nonzero.

\begin{center}
\begin{picture}(200,205)(-100,-90)
\thicklines
\put(80,8){$\Trop(X)$}
\put(67,65){$\nu(a) < 0$}
\put(50,85){$\nu(a) = 0$}
\put(30,103){$\nu(a) > 0$}
\put(-55,-82){$\Trop(X')$}
\color{navy}
\put(0,0){\line(-1,-1){70}}
\put(0,0){\line(1,0){100}}
\put(0,0){\line(0,1){100}}
\color{darkred}
\put(-60,-70){\line(1,2){83}}
\put(-35,-70){\line(1,2){77}}
\put(-10,-70){\line(1,2){69}}
\end{picture}
\end{center}

\noindent If $\nu(a)$ is positive then $\Trop(X)$ meets $\Trop(X')$ at two points, each with multiplicity $1$.   In this case, Theorem~\ref{thm:main plus multiplicities} says that $X$ meets $X'$ at two points, each with multiplicity 1, and one of these intersection points lies over each of the points of $\Trop(X) \cap \Trop(X')$.  It is also straightforward to check this directly.  If $\nu(a)$ is negative, then $\Trop(X)$ meets $\Trop(X')$ at a single point, but with tropical intersection multiplicity 2.  In this case, Theorem~\ref{thm:main plus multiplicities} says that $X$ meets $X'$ at either a single point with multiplicity 2, or at two points of multiplicity 1, and one can check that the intersection is always two points of multiplicity 1.  In both of these cases, the nonemptiness of $X \cap X'$ also follows from the transverse tropical lifting result of \cite{BJSST}, since the tropicalizations meet properly in the interiors of maximal faces.

Suppose $\nu(a)$ is zero.  Then $\Trop(X)$ and $\Trop(X')$ meet at a single point with tropical multiplicity 2, but the intersection is in a nonmaximal face of $\Trop(X)$, so transverse lifting results do not apply.  Nevertheless, Theorem~\ref{thm:main plus multiplicities} still says that $X$ and $X'$ meet at either a single point with multiplicity 2, or at two points of multiplicity 1.  Either possibility can occur; the algebraic intersection is a single point of multiplicity 2 when the characteristic is not 2 and $a$ is equal to $-1/4$, and two points of multiplicity 1 otherwise.
\end{ex2}

In the following example, the tropicalizations meet nonproperly along a positive dimensional set that does not contain the the tropicalization of any curve.

\begin{ex2}\label{ex:lines} 
Inside $Y = (K^*)^2$, let $X$ and $X'$ be given by $y = x + 1$ and $y = ax + b$, respectively, with $a$ and $b$ in $K^*$.  Assume $a$ and $b$ are not both 1, so the closures of $X$ and $X'$ are distinct lines in $K^2$.  In particular, $X$ and $X'$ intersect in at most one point.

Suppose $\nu(a)$ is zero and $\nu(b)$ is positive.  Then $\Trop(X)$ and $\Trop(X')$ intersect nonproperly, along the ray $\RR_{\leq 0} \cdot (1,1)$.

\begin{center}
\begin{picture}(200,160)(-100,-60)
\thicklines
\put(-28,2){$(0,0)$}
\put(65,-20){$\Trop(X)$}
\put(75,30){$\Trop(X')$}
\color{navy}
\put(0,0){\line(-1,-1){60}}
\put(0,0){\line(1,0){100}}
\put(0,0){\line(0,1){100}}
\color{darkred}
\put(20,20){\line(-1,-1){80}}
\put(20,20){\line(1,0){80}}
\put(20,20){\line(0,1){80}}
\end{picture}
\end{center}

\noindent If $a$ is 1, then the closures of $X$ and $X'$ are parallel lines in $K^2$, so none of the tropical intersection points lift.

Suppose $a$ is not 1.  Then the unique algebraic intersection point is
\[
X \cap X' = ( (1-b)/(a-1), (a-b)/(a-1)),
\]
and the unique point of $\Trop(X) \cap \Trop(X')$ that lifts is $(-\nu(a-1), -\nu(a-1))$.  For a suitable choice of $a$ congruent to 1 modulo $\fm$, any nonzero $G$-rational point in the tropical intersection can lift.  For such $a$, the initial degenerations $X_w$ and $X'_w$ coincide for all nonzero $w$ in $\Trop(X) \cap \Trop(X')$, and are disjoint otherwise.  In particular, none of the initial degenerations meet properly.  On the other hand, if $a$ is not congruent to $1$ modulo $\fm$, then the initial degenerations $X_w$ and $X'_w$ meet transversely at a single point for $w = (0,0)$, and are disjoint otherwise.  In this case $(0,0)$ is the unique tropical intersection point that lifts, as it must be by Theorem~\ref{exploded main}.  Note that, even though $\Trop(X \cap X')$ and the stable tropical intersection are both zero-dimensional, neither is necessarily contained in the other.
\end{ex2}

In the remaining examples, we assume the characteristic of $K$ is not 2 and consider tropicalizations of skew lines inside a smooth quadric surface.  These examples demonstrate the necessity of requiring the point of proper intersection to be a simple point of $\Trop(Y)$ or a smooth point of $Y_w$, in Theorems~\ref{main}, \ref{thm:lift-to-degens}, and \ref{exploded main}.

\begin{ex2}\label{ex:lines-in-quadric-1} 
Let $Y$ be the surface in $(K^*)^3$ given by
\[
z^2 + xy + x + y = 0.
\]
Then $Y$ contains the curves $X$ and $X'$ given by $x + 1 = z-1 = 0$ and $x + 1 = z+ 1 = 0$, respectively.  Now, $\Trop(X)$ and $\Trop(X')$ still meet properly at $w = (0,0,0)$, as shown.
\begin{center}
\begin{picture}(225,195)(-20,0)
\put(0,0){\resizebox{3in}{2.6in}{\includegraphics{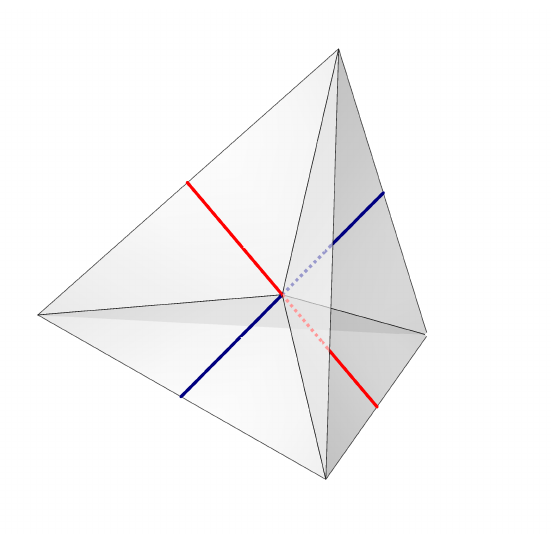}}}
\put(-15,107){$\Trop(Y)$}
\put(154,35){$\Trop(X')$}
\put(157,125){$\Trop(X)$}
\put(93,75){$w$}
\end{picture}
\end{center}
However, this tropical intersection point does not lift because the closures of $X$ and $X'$ are skew lines in $K^3$.  This is consistent with Theorems~\ref{main} and \ref{thm:lift-to-degens} because $w$ is not in the relative interior of a facet, and the initial degenerations $X_w$ and $X'_w$ are disjoint.  The closures of $X_w$ and $X'_w$ in $k^3$ are skew lines in the closure of $Y_w$, which is a smooth quadric surface.
\end{ex2}

\begin{ex2}\label{ex:lines-in-quadric-2} 
Let $a \in K^*$ be an element of positive valuation, and let $Y'$ be the surface given by
\[
z^2 -1 + a(xy + x + y + 1)= 0,
\]
and let $X$ and $X'$ be as in Example \ref{ex:lines-in-quadric-1}. Then $X$ and $X'$ are contained in $Y'$, and their tropicalizations still meet properly in $\Trop(Y')$ at the origin $w$, which is in the relative interior of a facet $\sigma$ of $\Trop(Y')$, as shown.
\begin{center}
\begin{picture}(225,195)(0,-10)
\put(0,0){\resizebox{3in}{2.5in}{\includegraphics{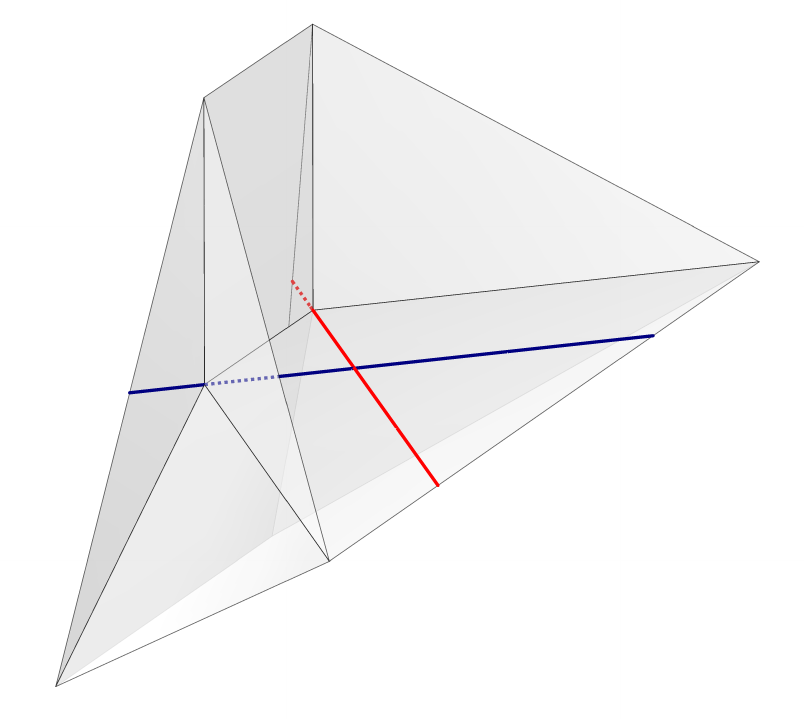}}}
\put(192,93){$\Trop(X)$}
\put(115,45){$\Trop(X')$}
\put(0,140){$\Trop(Y')$}
\put(90,80){$w$}
\put(146,70){$m(\sigma) = 2$}
\end{picture}
\end{center}
\noindent The tropical intersection point $w$ does not lift to $X \cap X'$, but this is still consistent with Theorems~\ref{main} and \ref{thm:lift-to-degens} because the multiplicity of the facet $\sigma$ is 2, and hence $w$ is not a simple point of $\Trop(Y')$.  The initial degeneration $Y'_w$ has two disjoint components, each of which contains the initial degeneration of one of the curves.
\begin{center}
\begin{picture}(225,195)(-10,0)
\put(0,0){\resizebox{3in}{2.6in}{\includegraphics{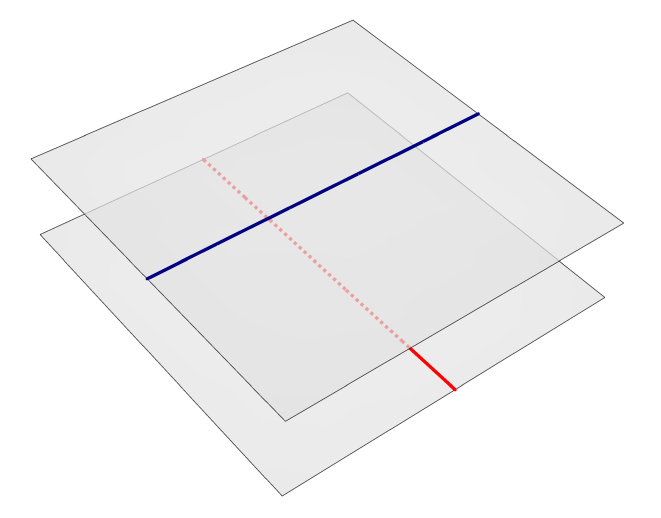}}}
\put(162,150){$X_w$}
\put(153,37){$X'_w$}
\put(-16,113){$Y'_w$}
\end{picture}
\end{center}
In particular, this tropical intersection point does not lift even to the intersection of the initial degenerations.
\end{ex2}

\begin{ex2}\label{ex:lines-in-quadric-3} 
Let $a \in K^*$ be an element of positive valuation, with $Y''$ be the surface in $(K^*)^3$ given by
\[
(x+1)(y+1)+(x+z)(y+z+a)=0.
\]
Let $X$ again be as in Example \ref{ex:lines-in-quadric-1}, and let $X''$ be the curve given by $y + 1 =  z - 1+a = 0$. Once again, $\Trop(X)$ and $\Trop(X'')$ meet properly at $w = (0,0,0)$.  Furthermore, the initial degenerations $X_w$ and $X''_w$ meet properly at a single point in $Y''_w$.  
\begin{center}
\begin{picture}(225,195)(-10,-10)
\put(10,0){\resizebox{2.5in}{2.5in}{\includegraphics{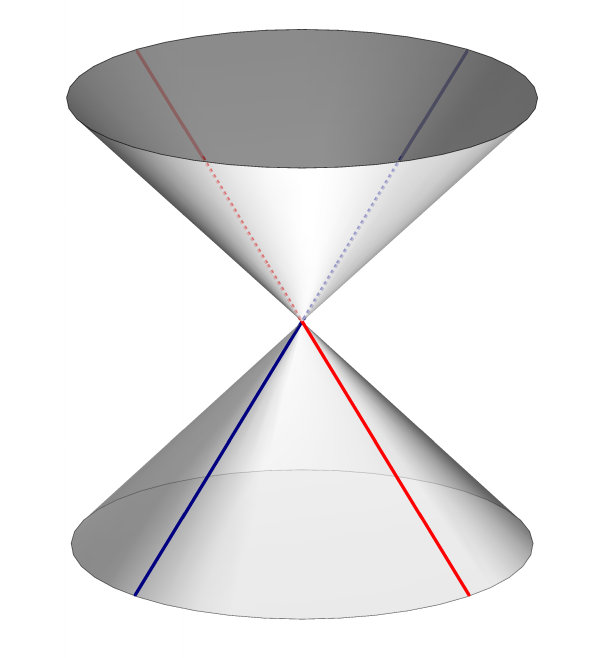}}}
\put(20,130){$Y''_w$}
\put(40,5){$X_w$}
\put(152,5){$X''_w$}
\end{picture}
\end{center}
This intersection point lifts to the initial degenerations but not to the general fiber, because the closures of $X$ and $X''$ are skew lines in $K^3$.  This is still consistent with Theorem~\ref{exploded main} because $Y''_w$ is a cone, and $X_w$ meets $X''_w$ at the singular point.
\end{ex2}

\appendix

\section{Initial degenerations, value groups, and base change}  \label{app:finite type}

Here we study how initial degenerations behave with respect to arbitrary extensions of valued fields.  These basic results are used throughout the paper to reduce our main lifting theorems to the case where $w$ is in $N_G$, by extending the ground field.

Let $M_{G,w}$ be the maximal sublattice of $M$ on which $w$ is $G$-rational.  In other words,
\[
M_{G,w} = \{ u \in M \ | \ \<u,w\> \mbox{ is in } G \}.
\]
Since $K$ is algebraically closed, the value group $G$ is divisible, and hence $M_{G,w}$ is saturated in $M$.  If the weight of a monomial $a x^u$ is zero, then $\<u, w\> = -\nu(a)$, and hence $u$ is in $M_{G,w}$.  In particular, $M_{G,w}$ contains all exponents of monomials that restrict to nonzero functions on $T_w$, and $T_w$ is a torsor over the torus associated to $M_{G,w}$.

\begin{prop2}  \label{prop:finite type}
Suppose the valuation is nontrivial.  Then the integral model $\cT^w$ is of finite type over $\Spec R$ if and only if $w$ is in $N_G$.  Furthermore, if $w$ is in $N_G$, then $\cT^w$ is of finite presentation over $\Spec R$.
\end{prop2}

\begin{proof}
Suppose $w$ is in $N_G$.  Let $u_1, \ldots, u_r$ be a basis for $M$, and choose $a_1, \ldots, a_r$ in $K^*$ such that $\nu(a_i) = -\<u_i, w\>$.  Then $R[M]^w$ is generated over $R$ by
\[
\big \{a_1 x^{u_1}, (a_1 x^{u_1})^{-1}, \ldots, a_r x^{u_r}, (a_r x^{u_r})^{-1} \big \},
\]
and hence is of finite type. Furthermore, the relations are generated by 
\[
\{1- a_i x^{u_i} \cdot (a_i x^{u_i})^{-1}\},
\]
so $R[M]^w$ is of finite presentation.

For the converse, suppose $w$ is not in $N_G$.   Then, $M_{G,w}$ is a proper sublattice of $M$, and hence the special fiber $T_w$ has dimension strictly less than the dimension of $T$.  Since $\cT^w$ is irreducible and fiber dimension is semicontinuous in irreducible families of finite type, it follows that $\cT^w$ is not of finite type.
\end{proof}

\begin{rem2}
Suppose the valuation is trivial, and let $\cS_w$ be the set of lattice points 
$u \in M$ such that $\langle u,w\rangle \geq 0$.  Then $\cS_w$ is a 
subsemigroup of $M$, and $R[M]^w$ is naturally identified with the semigroup 
ring $K[\cS_w]$.  The semigroup $\cS_w$ is finitely generated if and only if the 
ray spanned by $w$ has rational slope, and it follows that $\cT^w$ is of 
finite type over $K$ if and only if $w$ is a scalar multiple of a lattice 
point.
\end{rem2}

When the schemes $\cT^w$ and $\cX^w$ are not of finite type, there are many technical difficulties in handling them directly.  These technical difficulties can be overcome by extending scalars, since these schemes become finite type after a suitable base change, as follows.

Suppose the value group $G$ is a proper subgroup of $\RR$, and $b$ is a real number that is not in $G$.  Then there is a unique valuation $\tilde \nu$ on the function field $K(t)$ such that $\tilde \nu(t) = b$.  This valuation extends to the algebraic closure $\widetilde K$ of $K(t)$.  Iterating this procedure finitely many times, we can ensure that an arbitrary $w$ is rational over the value group of a suitable extension of $K$.  In particular, for a suitable choice of extension $\widetilde K | K$, the scheme $\widetilde \cT_w$ is of finite type over the valuation ring $\widetilde R$.

Let $\widetilde K | K$ be an arbitrary extension of valued fields, and let $\widetilde G$ be the value group of $\widetilde K$, with $\widetilde R$ the valuation ring in $\widetilde K$.  For any closed subscheme $X$ of $T$ over $K$, the tropicalization of the base change $\widetilde X$ is equal to the tropicalization of $X$ \cite[Proposition~6.1]{analytification}, so the initial degeneration $\widetilde X_w$ is nonempty if and only if $X_w$ is nonempty.  Here we give a more precise geometric relationship between these initial degenerations.  First we treat the associated schemes over the valuation rings.  Let
\[
\varphi: \widetilde \cT_w \rightarrow \cT_w
\]
be the natural map induced by the inclusion of tilted group rings, which is equivariant over the projection of tori induced by the inclusions $M_{G,w} \hookrightarrow M_{\widetilde G,w}$ and $R \hookrightarrow \widetilde R$.

\begin{thm2}  \label{thm:integral base change}
The scheme $\widetilde \cX^w$ is the preimage of $\cX^w$ under $\varphi$.
\end{thm2}

\begin{proof}
It is clear that $\varphi$ maps $\widetilde \cX^w$ into $\cX^w$.  To show that $\widetilde \cX^w$ is the full preimage of $\cX^w$, we must prove that any Laurent polynomial in the tilted group ring $\widetilde R[M]^w$ that vanishes on $\widetilde X$ is in the ideal generated by $I_X \cap R[M]^w$.  

Let $f = \sum_u \alpha(u)x^u$ be a nonzero Laurent polynomial over $\widetilde K$ in $I_{\widetilde X}$.  Then $f$ can be written as
\[
f = \alpha_1 f_1 + \cdots + \alpha_n f_n,
\]
with $f_1, \ldots, f_n$ in $I_X$ and linearly independent over $K$, and $\alpha_i$ in $\widetilde K$ all nonzero.  Say $a_i(u)$ is the coefficient of $x^u$ in $f_i$, so the coefficient of $x^u$ in $f$ is
\[
\alpha(u) = \alpha_1 a_1 (u) + \cdots + \alpha_n a_n(u).
\]
Now, suppose $f$ is in $\widetilde R[M]^w$.  If each $\alpha_i f_i$ is in $\widetilde R[M]^w$, then it is easy to see that $f$ is in the ideal generated by $I_X \cap R[M]^w$, by applying the case $n = 1$, below.  The difficulty is that there may be some cancellation of leading terms in the above expression for $f$, by which we mean that the valuation of some coefficient $\alpha(u)$ may be strictly larger than $\min_i \{ \widetilde \nu (\alpha_i a_i(u)) \}$.  Roughly speaking, this means that the vector $(\alpha_1, \ldots, \alpha_n)$ is nearly orthogonal to $(a_1(u), \ldots, a_n(u))$.

When $n$ is greater than $1$, we proceed by carefully eliminating one term in the summation, writing $f$ as a $\tilde K$-linear combination of $f_1, ..., f_{n-1}$ plus a single element of $\widetilde K$ times a $K$-linear combination of $f_1, ..., f_n$.  A suitable choice in the elimination procedure ensures that each of these two terms is in $\widetilde R [M]^w$, and we deduce that they are in the ideal generated by $I_X \cap R[M]^w$, by induction on $n$.  Roughly speaking, we choose $u_1, \ldots, u_{n-1}$ so that $(a_1(u_i), \ldots, a_n(u_i))$ are as close to orthogonal as possible to $(\alpha_1, \ldots, \alpha_n)$, for $1 \leq i \leq n-1$.  Then we replace $(\alpha_1, \ldots, \alpha_n)$ by the unique vector in $K^n$ that is orthogonal to $(a_1(u_i), \ldots, a_n(u_i))$ for $1\leq i \leq n-1$ and with $n$th coordinate $1$.  The corresponding $K$-linear combination of $f_1, \ldots, f_n$ is then multiplied by $\alpha_n$ and subtracted from $f$.  The details are as follows.

Suppose $n = 1$.  Let $m = a_1(u) x^u$ be a monomial of lowest weight in $f_1$, and let $g = f_1 / m$.  Then $f$ can be expressed as
\[
f = \alpha_1 m g,
\]
with $g$ in $I_X \cap R[M]^w$, and $\alpha_1 m$ in $\widetilde R [M]^w$, as required.

We proceed by induction on $n$.   Given $u_1, \ldots, u_{n-1}$ in $M$, consider the matrix whose $(i,j)$th entry is $a_i(u_j)$, and let $\delta_i$ be the $i$th maximal minor
\[ 
\delta_i = 
\begin{vmatrix}
a_1(u_1) & \cdots & a_1(u_{n-1}) \\
\vdots      & 		 & \vdots 	     \\
\widehat{a_i(u_1)} & \cdots & \widehat {a_i(u_{n-1})} \\
\vdots & 			& \vdots 			\\
a_n(u_1) & \cdots & a_n(u_{n-1}) 
\end{vmatrix}.
\]
Since $f_1, \ldots, f_{n-1}$ are linearly independent over $K$, we may choose $u_1, \ldots, u_{n-1}$ so that $\delta_n$ is nonzero and
\[
\tilde \nu(\alpha(u_1)) + \cdots + \tilde \nu(\alpha(u_{n-1})) - \nu( \delta_n)
\]
is as large as possible.  This choice is essential in the proof of the following claim.

We claim that
\[
h = \frac{\alpha_n}{\delta_n} \sum_{i=1}^n ( -1)^{n-i} \delta_i f_i 
\]
is in $\widetilde R [M]^w$.  The claim implies that $h$ is in the ideal generated by $I_X \cap R[M]^w$, by the case $n =1$, above, and also that the difference $f - h$ is in $\widetilde R[M]^w$.  The coefficient of $f_n$ in the above expression is equal to $\alpha_n$, by construction, so the difference $f - h$ can be written as a $\widetilde K$-linear combination of $f_1, \ldots, f_{n-1}$.  It follows by induction that $f-h$ is also in the ideal generated by $I_X \cap R[M]^w$, and this proves the theorem.

It therefore remains to show that $h = \sum \beta(u)x^u$ is in $\widetilde R[M]^w$, which means that the $w$-weight of each monomial $\beta(u) x^u$ is nonnegative.  Fix one such monomial, write $u_n = u$ for its exponent, and let $A$ be the square $n \times n$ matrix whose $(i,j)$th entry is $a_i(u_j)$,
\[
A = 
\left( \begin{matrix}
a_1(u_1) & \cdots & a_1(u_n) \\
\vdots & \ddots & \vdots \\
a_n(u_1) & \cdots & a_n(u_n) 
\end{matrix} \right).
\]
Let $A_{ij}$ be the $(i,j)$th minor of $A$, the determinant of the submatrix obtained by deleting the $i$th row and $j$th column.  So $\delta_i = A_{in}$.  Expanding $\det A$ in the last column shows that
\[
\beta(u_n) = \frac{\alpha_n}{A_{nn}} \det A.
\]
Therefore, the valuation of $\beta(u_n)$ is
\[
\tilde \nu(\beta(u_n)) = \nu ( \det A) - \nu (A_{nn}) + \tilde \nu(\alpha_n).
\]

Since $f$ is in $\widetilde R[M]^w$ by hypothesis, it will be enough to show that $\tilde \nu(\beta(u_n))$ is at least as large as $\tilde \nu(\alpha(u_n))$.  To compare $\beta(u_n)$ with $\alpha(u_n)$, we consider the matrix
\[
A' = 
\left( \begin{matrix}
\alpha_1a_1(u_1) & \cdots & \alpha_1 a_1(u_n) \\
\vdots & \ddots & \vdots \\
\alpha_{n-1}a_{n-1}(u_1) & \cdots & \alpha_{n-1}a_{n-1}(u_n)\\
\alpha(u_1) & \cdots & \alpha(u_n) 
\end{matrix} \right).
\]
Recall that the coefficient $\alpha(u_j)$ in the bottom row is $\alpha_1 a_1(u_j) + \cdots + \alpha_n a_n (u_j)$, so the determinant of $A'$ is $\alpha_1 \cdots \alpha_n \det A$.  Expanding $\det A'$ in the last row gives also
\[
\det A' = \alpha_1 \cdots \alpha_{n-1} \bigg( \sum_{i=1}^n (-1)^{n-i} \alpha(u_i) A_{ni} \bigg),
\]
and comparing these two expressions for $\det A'$ yields
\[
\det A = \frac{1}{\alpha_n} \bigg( \sum_{i=1}^n (-1)^{n-i} \alpha(u_i) A_{ni} \bigg),
\] 
Therefore,
\[
\nu(\det A) \geq \min_i \{ \nu(\alpha(u_i)) + \nu(A_{ni})\} - \nu(\alpha_n).
\]
Now $u_1, \ldots, u_{n-1}$ were chosen so that this minimum occurs at $i = n$.  Substituting the resulting inequality for $\nu(\det A)$ into the expression for $\tilde \nu (\beta(u_n))$ above shows that $\tilde \nu(\beta(u_n))$ is greater than or equal to $\tilde \nu (\alpha(u_n))$.  This proves the claim, and the theorem follows.
\end{proof}

We now pass to the initial degenerations.  Let
\[
\phi: \widetilde T_w \rightarrow T_w
\]
be the natural projection of torus torsors induced by the inclusions of tilted group rings, modulo monomials of strictly positive $w$-weight.

\begin{thm2}  \label{thm:initial base change}
The initial degeneration $\widetilde X_w$ is the preimage of $X_w$ under $\phi$.
\end{thm2}

\begin{proof}
It is clear that $\phi$ maps $\widetilde X_w$ into $X_w$.  Any element $f_w$ of the ideal of $\widetilde X_w$ is the residue of a Laurent polynomial $f$ in $I_{\widetilde X} \cap \widetilde R[M]^w$, which is then in the ideal generated by $I_X \cap R[M]^w$, by Theorem~\ref{thm:integral base change}.  Taking residues shows that $f_w$ in the ideal generated by the pullback of $I_{X_w}$, and it follows that $\widetilde X_w$ is the full preimage of $X_w$.
\end{proof}

\begin{rem2} \label{rmk:properties of base change}
Since $\phi$ is smooth and has connected fibers, it follows that many geometric properties of initial degenerations are preserved under extensions of valued fields.  For instance, the sum of the multiplicities of the irreducible components of $X_w$ is equal to that of $\widetilde X_w$, which is helpful for defining tropical multiplicities.  Most importantly for our purposes, a point $\tilde x$ is smooth in $\widetilde Y_w$ if and only if $\phi(\tilde x)$ is smooth in $Y_w$, $\widetilde X_w$ meets $\widetilde X'_w$ properly at $\tilde x$ if and only if $X_w$ meets $X'_w$ properly at $\phi(\tilde x)$, and $\tilde x$ is in $(\widetilde X \cap \widetilde X')_w$ if and only if $\phi(\tilde x)$ is in $(X \cap X')_w$.
\end{rem2}

\section{Topology of finite type morphisms}\label{app:closed-pts}

Because the results may be of independent interest, we explain how Theorem~\ref{thm:closed-pts} on the existence of closed points in fibers, and Theorem
\ref{thm:lift to K} on lifting points of intersection, can both be 
extended to an arbitrary base scheme.

The proof of Proposition~\ref{prop:krull} does not use that the base scheme is the spectrum of a valuation ring of rank $1$, and in fact yields the following result.

\begin{prop2}\label{prop:app-krull} Let $X \to S$ be a flat morphism of finite 
type of irreducible schemes, and
suppose that $D$ is a locally principal closed subscheme of $X$ that does 
not meet the generic fiber.  Then, for every $s \in S$, every irreducible 
component of the fiber $D_s$ is an irreducible component of $X_s$.
\end{prop2}

\begin{rem2}\label{rem:blowup}
Some hypothesis such as flatness is necessary for such results on locally principal subschemes over general base schemes, as shown by the following example.  Suppose $X \rightarrow S$ is the blowup of the affine plane at the origin.  Then the strict transform $D$ of a line through the origin is locally principal and does not meet the generic fiber, but its intersection with the exceptional fiber is a single point.
\end{rem2}

We can now prove the first stated result.

\begin{thm2}\label{thm:lift-pts-general}
Let $X \rightarrow S$ be a morphism locally of finite type, and let $s$ be a specialization of $s'$ in $S$.  Suppose $x'$ is a point in $X_{s'}$ specializing to a closed point $x$ in $X_s$.  Then there is a closed point $x''$ in $X_{s'}$ which specializes to $x$, and such that $x'$ specializes to $x''$. Moreover, the set of such $x''$ is Zariski dense in the closure of $x'$ inside $X_{s'}$.  More generally, if $x$ is not necessarily closed in $X_s$, and $k(x)$ denotes the residue field of $x$, we can choose $x''$ to satisfy the inequality 
\begin{equation}\tag{B.1}\label{eq:trans-ineq} 
\trdeg k(x'')/k(s') \leq \trdeg k(x)/k(s),
\end{equation}
and again the choices of $x''$ are Zariski dense in the closure of $x'$ inside $X_{s'}$. Moreover, if $S$ is the spectrum of a valuation ring, we have equality in \eqref{eq:trans-ineq}.
\end{thm2}

Note that the example of Remark \ref{rem:blowup} shows that we cannot do 
better than the inequality \eqref{eq:trans-ineq} for a general base scheme.

\begin{proof} First, in light of the generalized Proposition 
\ref{prop:app-krull}, the argument of Theorem \ref{thm:closed-pts}
goes through to prove the desired result in the case that $S$
is an arbitrary valuation ring, with $s'$ the generic point of $S$.
The only subtlety is that in the general case, the components of $\cD$
not containing $x$ do not necessarily form a closed subset. Nonetheless,
we can pass to an open neighborhood of $x$ such that every component of 
$\cD$ meeting the generic fiber (if there are any) must contain
$x$, and this suffices for the argument. To remove the restriction that $s'$
be the generic point of $S$, we simply note that the closure of $s'$ is again
the spectrum of a valuation ring.

We thus wish to reduce to the valuation ring case. Replacing $X$ by the 
closure of $x'$, we may assume $X$ is integral and $x'$ is its generic point.
By \cite[Proposition~7.1.4(ii)]{EGA2}, there is a valuation ring $A$ in 
$k(x')$ with a dominant morphism
\[
\Spec A \rightarrow X
\]
mapping the closed point to $x$.  Let $R$ be the valuation ring in the residue field $k(s')$ given by intersecting with $A$ in $k(x')$, and let $X' = X \times_S \Spec R$.  Then the following diagram is commutative.
\[
\xymatrix{{X'}\ar[r]\ar[d] & {X} \ar[d] \\
{\Spec R}\ar[r] & {S}}
\]
By construction the map from $\Spec A$ to $X$ factors through $X'$.  So, by the case we have already handled, there is a point $x''$ in the generic fiber of $X'$ over $\Spec R$ specializing to the image $\tilde{x}$ in $X'$ of the closed point of $\Spec A$ and satisfying the desired equality of residue field extensions; moreover, such points are Zariski dense in the generic fiber.  Since $k(s')$ is identified with the fraction field of $R$, by construction, the generic fiber of $X'$ maps isomorphically to the generic fiber of $X$. Finally, observing that we have the inequality
\[ \trdeg k(\tilde{x})/(R/\fm_R) \leq \trdeg k(x)/k(s),\]
we conclude the desired statement. 
\end{proof}

\smallskip

Naive statements about global codimension and subadditivity do not extend from valuation rings of rank 1 to valuation rings of higher rank, as shown by Example~\ref{ex:not subadd}.  Nevertheless, hypotheses on codimension of intersection can still yield lifting results even when subadditivity of codimension fails, as demonstrated by the following theorem.

\begin{thm2} \label{thm:lift general}
Let $Y \rightarrow S$ be a smooth morphism, and let $X$ and $X'$ be closed subschemes of $Y$, flat over $S$, such that the codimension in $Y$ of every component of $X$ and every component of $X'$ is less than or equal to $j$ and $j'$, respectively.  Suppose that, for some $s \in S$, the fibers $X_s$ and $X'_s$ meet in codimension $j+j'$ at a point $x$ in $Y_s$.  Then for any $s' \in S$ specializing to $s$:
\begin{enumerate}
\item There is a point $x'$ in $X_{s'} \cap X'_{s'}$ specializing to $x$.
\item If $x$ is closed in $Y_s$, then $x'$ may be chosen to be closed in
$Y_{s'}$. More generally, we may choose $x'$ so that we have
\[
\trdeg k(x')/k(s') \leq \trdeg k(x)/k(s).
\]
\end{enumerate}
\end{thm2}

\begin{proof} 
In light of Theorem~\ref{thm:lift-pts-general}, the second assertion follows immediately from the first. Observe that the flatness hypotheses mean that the hypotheses of the theorem are preserved under arbitrary base change. In particular, we reduce to the case that $S$ is the spectrum of a valuation ring, with $s'$ the generic point. We then prove the desired statement with a reduction to the diagonal and inductive application of Proposition \ref{prop:app-krull}, using that $S$ is the spectrum of a valuation ring to preserve the flatness hypothesis.
\end{proof}

\begin{rem2}
In Example~\ref{ex:not subadd}, the special fibers of the two subschemes coincide, and the generic fibers are disjoint.  This does not contradict Theorem~\ref{thm:lift general} because the special fibers do not meet properly, and the failure of subadditivity is for simple numerical reasons.  The intersection that does not lift has dimension one larger than expected, but is supported in a fiber of codimension $r$, which is greater than 1.
\end{rem2}

\section{An application to tropical elimination theory}
\label{Appendix}

Let $X$ be an irreducible closed subscheme of $T$ and let $\varphi: T \rightarrow T'$ be a homomorphism of tori that induces a generically finite morphism from $X$ to the closure of its image, which we denote $X'$.  Then, set theoretically, $\Trop(X')$ is the image of $\Trop(X)$ under the induced linear map $\phi: N_\RR \rightarrow N'_{\RR'}$.  The fundamental problem of tropical elimination theory, solved by Sturmfels and Tevelev for the special case where the valuation is trivial, is to determine the multiplicities on the facets of $\Trop(X')$.  Here we use tropical lifting theorems to generalize \cite[Theorem~1.1]{SturmfelsTevelev08} to the case of a nontrivial valuation.  See also \cite[Section~8]{BPR11} for an analytic proof of this result and applications to curves.

After subdividing, we may assume that $\phi$ maps each face of $\Trop(X)$ onto a face of $\Trop(X')$.

\begin{thm2} \label{thm:sturmfelstevelev}
  The multiplicity of a facet $\sigma'$ in $\Trop(X')$ is
\[
m(\sigma') = \frac{1}{\delta} \sum_{\phi(\sigma) = \sigma'} m(\sigma) \cdot [N'_{\sigma'} : \phi(N_\sigma)],
\]
where $\delta$ is the degree of $\varphi$.
\end{thm2}

\noindent Theorem~\ref{thm:sturmfelstevelev} and Corollary~\ref{cor:cycles} together imply that tropicalization of cycles commutes with push forward.  See \cite[Theorem~13.17]{Gubler12}.

\begin{proof} We prove the theorem by intersecting $X$ and $X'$ with suitable
translates of subtori and then counting points using tropical intersection
theory and lifting theorems.  First, we choose the translated subtori to
ensure that these intersections occur in a locus where $\alpha$ is
well-behaved.

Since $\varphi$ is generically finite, there is a dense open subset $U'
\subset X'$ such that the induced map $\varphi^{-1}(U') \rightarrow U'$ is
finite of degree $\delta$ \cite[Exercise II.3.7]{Hartshorne77}.
Shrinking $U'$ further, if necessary, we may assume that $U'$ is smooth
and $\varphi^{-1}(U')$ is flat over $U'$, so the preimage of a
zero-dimensional subscheme of length $m$ in $U$ is a zero-dimensional
subscheme of length $\delta \cdot m$.

Let $\Lambda'$ be a sublattice of  $N'$ complementary to $N'_{\sigma'}$, so $N'$ splits as a direct sum 
\[
N' = N'_{\sigma'} \oplus \Lambda'.
\]  
We write $T'_{\Lambda'}$ for the subtorus of $T'$ whose lattice of one-parameter subgroups is $\Lambda'$.  
Let $\Lambda \subset N$  be the preimage of $\Lambda'$, with $T_\Lambda$ the associated subtorus of $T$.  
Let $\widetilde T_\Lambda$ be the preimage of $T'_{\Lambda'}$ which is the product of $T_\Lambda$ with a zero-dimensional scheme of length
\[
\ell = \frac{ [N' : \phi(N)]} {[\Lambda' : \phi(\Lambda)]}.
\]
If the characteristic of $K$ is zero, or prime to $\ell$, then $\widetilde T_\Lambda$ is a union of translates of $T_\Lambda$ by $\ell$ 
distinct torsion points.

We claim that, for any nonempty open subset $U \subset X$ there is an open dense set of $t \in T$ such that $t\widetilde T_\Lambda \cap X$ 
is contained in $U$.  
To see this, consider the incidence subscheme $W$ in $T \times X$ parametrizing pairs $(t,x)$ such that $x$ is in $t \widetilde T_\Lambda$.  
Then the first projection is dominant and generically finite, while the second projection is flat and maps $W$ surjectively onto $X$.  
Therefore, the preimage of $X \smallsetminus U$ has positive codimension in $W$ and hence projects into a set of positive codimension 
in $T$.  
Therefore the complement of the closure of $p_1 (p_2^{-1} (X \smallsetminus U))$ is an open dense subset of $T$ consisting of points $t$ 
such that $t \widetilde T_\Lambda \cap X$ is contained in $U$.  

We fix
\[
U = \varphi^{-1}(U') \cap X^{\mathrm{sm}}
\]
and choose $v \in N_G$ such that $v' = \phi(v)$ is in the relative interior of $\sigma'$.  Since $\Trop^{-1}(v)$ is Zariski dense in $T$, we can choose $t \in \Trop^{-1}(v)$ in the open dense subset of $T$
described above, such that $t \widetilde T_\Lambda \cap X$ is contained in $U$.
Note that, since $v'$ is in the relative interior of the maximal face  $\sigma'$ of $\Trop(X')$, 
it has finitely many preimages $v_1, \ldots, v_r$ in $\Trop(X)$, one in
each maximal face $\sigma_i$ mapping onto $\sigma'$.  Let $\sigma_i$ be
the maximal face of $\Trop(X)$ containing $v_i$.

Let $t' = \varphi(t)$.  We now consider $t'T'_{\Lambda'} \cap X'$ and $t \widetilde T_\Lambda \cap X$, and especially the parts of these 
intersections that live in $\Trop^{-1}(v')$ and $\Trop^{-1}(v_i)$, for $1 \leq i \leq r$, respectively.  
By construction, $\Trop(t'T'_{\Lambda'})$ is the affine linear space $\Lambda'_\RR + v'$ with multiplicity 1, and meets $\Trop(X')$ transversally 
at $v'$.  The translation of $t'T'_{\Lambda'}$ by a sufficiently small vector in $N'_\RR$ also meets $\sigma'$ transversally 
at a single point, so the fan displacement rule gives the local tropical 
intersection multiplicity as
\[
i(v', \Trop(X') \cdot \Trop(t'T'_{\Lambda'})) = m(\sigma') \cdot [N' : N'_{\sigma'} + \Lambda'].
\]
By the choice of $t$, the intersection of $t'T'_{\Lambda'}$ with $X'$ is contained in the smooth locus of $X'$.  
In particular, both are Cohen-Macaulay along their intersection in $\Trop^{-1}(v')$.  
Therefore, Corollary~\ref{cor:CM} says that the intersection 
of $X'$ with $t'T'_{\Lambda'}$ in $\Trop^{-1}(v')$ is a zero-dimensional scheme $Z'$ 
of length $i(v', \Trop(X') \cdot \Trop(t'T'_{\Lambda'}))$.  Similarly, $\Trop(t \widetilde T_\Lambda)$ is the affine linear space $\Lambda_\RR+ v$ 
with multiplicity $\ell$, and meets $\Trop(X)$ transversally at $v_i$ with local tropical intersection multiplicity
\[
i(v, \Trop(X) \cdot \Trop(t \widetilde T_\Lambda)) = \ell \cdot m(\sigma_i) \cdot [N : N_{\sigma_i} + \Lambda].
\]
Both $t \widetilde T_\Lambda$ and $X$ are smooth and hence Cohen-Macaulay along their intersection, so the intersection of $t \widetilde T_\Lambda$ with $X$ in 
$\Trop^{-1}(v_i)$ is a zero-dimensional scheme $Z_i$ of length $i(v, \Trop(X) \cdot \Trop(t \widetilde T_\Lambda))$.

By the choice of $t$, the map $\varphi$ is finite of degree $\delta$ in a neighborhood of $Z'$.  Furthermore, the preimage 
$\varphi^{-1}(Z')$ is exactly $Z_1 \cup \cdots \cup Z_r$.   Therefore,
\[
\length (Z')  = \frac{1}{\delta}( \length(Z_1) + \cdots + \length(Z_r)).
\]
Substituting the above tropical intersection multiplicities for these lengths gives the identity
\[
m(\sigma') \cdot [N' : N_{\sigma'} + \Lambda'] = \frac{\ell}{\delta} \cdot \sum_{i = 1}^r m(\sigma_i) \cdot [N: N_{\sigma_i} + \Lambda].
\]
Now $\sigma_1, \ldots, \sigma_r$ are exactly the faces of $\Trop(X)$ that map onto $\sigma$, and $\ell = [N' : \phi(N)] / [\Lambda' : \phi(\Lambda)]$.  
By rearranging terms, one then sees that, to prove the theorem, it suffices to show
\[
[N' : \phi(N)] \cdot  [N: N_{\sigma_i} + \Lambda]  = [N'_{\sigma'} : \phi(N_{\sigma_i})] \cdot [ \Lambda' : \phi(\Lambda)],
\]
for $1 \leq i \leq r$.  Both sides are equal to $[N' : \phi(N_{\sigma_i} +
\Lambda)]$, and the theorem follows.
\end{proof}

\bibliographystyle{amsalpha}
\bibliography{math}

\newcommand{\etalchar}[1]{$^{#1}$}
\providecommand{\bysame}{\leavevmode\hbox to3em{\hrulefill}\thinspace}
\providecommand{\MR}{\relax\ifhmode\unskip\space\fi MR }
\providecommand{\MRhref}[2]{%
  \href{http://www.ams.org/mathscinet-getitem?mr=#1}{#2}
}
\providecommand{\href}[2]{#2}
\begin{thebibliography}{BLdM11}

\bibitem[AR10]{AllermannRau10}
L.~Allermann and J.~Rau, \emph{First steps in tropical intersection theory},
  Math. Z. \textbf{264} (2010), no.~3, 633--670.

\bibitem[Arn73]{Arnold73}
J.~Arnold, \emph{Krull dimension in power series rings}, Trans. Amer. Math.
  Soc. \textbf{177} (1973), 299--304.

\bibitem[Ber90]{Berkovich90}
V.~Berkovich, \emph{Spectral theory and analytic geometry over
  non-{A}rchimedean fields}, Mathematical Surveys and Monographs, vol.~33,
  American Mathematical Society, Providence, RI, 1990.

\bibitem[BG84]{BieriGroves84}
R.~Bieri and J.~Groves, \emph{The geometry of the set of characters induced by
  valuations}, J. Reine Angew. Math. \textbf{347} (1984), 168--195.

\bibitem[BGR84]{BGR84}
S.~Bosch, U.~G{\"u}ntzer, and R.~Remmert, \emph{Non-{A}rchimedean analysis},
  Grundlehren der Mathematischen Wissenschaften, vol. 261, Springer-Verlag,
  Berlin, 1984.

\bibitem[BJS{\etalchar{+}}07]{BJSST}
T.~Bogart, A.~Jensen, D.~Speyer, B.~Sturmfels, and R.~Thomas, \emph{Computing
  tropical varieties}, J. Symbolic Comput. \textbf{42} (2007), no.~1-2, 54--73.

\bibitem[BLdM11]{BrugalleLopez11}
E.~Brugalle and L.~L\'{o}pez~de Medrano, \emph{Inflection points of real and
  tropical plane curves}, preprint, arXiv:1102.2478v2, 2011.

\bibitem[BLR90]{BLR90}
S.~Bosch, W.~L\"{u}tkebohmert, and M.~Raynaud, \emph{N\'eron models},
  Ergebnisse der Mathematik und ihrer Grenzgebiete (3), vol.~21,
  Springer-Verlag, Berlin, 1990.

\bibitem[BPR11]{BPR11}
M.~Baker, S.~Payne, and J.~Rabinoff, \emph{Nonarchimedean geometry,
  tropicalization, and metrics on curves}, preprint, arXiv:1104.0320v1, 2011.

\bibitem[Dra08]{Draisma08}
J.~Draisma, \emph{A tropical approach to secant dimensions}, J. Pure Appl.
  Algebra \textbf{212} (2008), no.~2, 349--363.

\bibitem[EH86]{EisenbudHarris86}
D.~Eisenbud and J.~Harris, \emph{Limit linear series: basic theory}, Invent.
  Math. \textbf{85} (1986), no.~2, 337--371.

\bibitem[FS97]{FultonSturmfels97}
W.~Fulton and B.~Sturmfels, \emph{Intersection theory on toric varieties},
  Topology \textbf{36} (1997), no.~2, 335--353.

\bibitem[Ful93]{Fulton93}
W.~Fulton, \emph{Introduction to toric varieties}, Annals of Mathematics
  Studies, vol. 131, Princeton University Press, Princeton, NJ, 1993.

\bibitem[Ful98]{IT}
\bysame, \emph{Intersection theory}, second ed., Ergebnisse der Mathematik und
  ihrer Grenzgebiete. 3. Folge. A Series of Modern Surveys in Mathematics,
  vol.~2, Springer-Verlag, Berlin, 1998.

\bibitem[Gro60]{EGA1}
A.~Grothendieck, \emph{\'{E}l\'ements de g\'eom\'etrie alg\'ebrique. {I}. {L}e
  langage des sch\'emas}, Inst. Hautes \'Etudes Sci. Publ. Math. (1960), no.~4,
  228.

\bibitem[Gro61a]{EGA2}
\bysame, \emph{\'{E}l\'ements de g\'eom\'etrie alg\'ebrique. {II}. \'{E}tude
  globale \'el\'ementaire de quelques classes de morphismes}, Inst. Hautes
  \'Etudes Sci. Publ. Math. (1961), no.~8, 222.

\bibitem[Gro61b]{EGA3.1}
\bysame, \emph{\'{E}l\'ements de g\'eom\'etrie alg\'ebrique. {III}. \'{E}tude
  cohomologique des faisceaux coh\'erents. {I}}, Inst. Hautes \'Etudes Sci.
  Publ. Math. (1961), no.~11, 167.

\bibitem[Gro63]{EGA3.2}
\bysame, \emph{\'{E}l\'ements de g\'eom\'etrie alg\'ebrique. {III}. \'{E}tude
  cohomologique des faisceaux coh\'erents. {II}}, Inst. Hautes \'Etudes Sci.
  Publ. Math. (1963), no.~17, 91.

\bibitem[Gro64]{EGA4.1}
\bysame, \emph{\'{E}l\'ements de g\'eom\'etrie alg\'ebrique. {IV}. \'{E}tude
  locale des sch\'emas et des morphismes de sch\'emas. {I}}, Inst. Hautes
  \'Etudes Sci. Publ. Math. (1964), no.~20, 259.

\bibitem[Gro65]{EGA4.2}
\bysame, \emph{\'{E}l\'ements de g\'eom\'etrie alg\'ebrique. {IV}. \'{E}tude
  locale des sch\'emas et des morphismes de sch\'emas. {II}}, Inst. Hautes
  \'Etudes Sci. Publ. Math. (1965), no.~24, 231.

\bibitem[Gro66]{EGA4.3}
\bysame, \emph{\'{E}l\'ements de g\'eom\'etrie alg\'ebrique. {IV}. \'{E}tude
  locale des sch\'emas et des morphismes de sch\'emas. {III}}, Inst. Hautes
  \'Etudes Sci. Publ. Math. (1966), no.~28, 255.

\bibitem[Gro67]{EGA4.4}
\bysame, \emph{\'{E}l\'ements de g\'eom\'etrie alg\'ebrique. {IV}. \'{E}tude
  locale des sch\'emas et des morphismes de sch\'emas {IV}}, Inst. Hautes
  \'Etudes Sci. Publ. Math. (1967), no.~32, 361.

\bibitem[Gub12]{Gubler12}
W.~Gubler, \emph{A guide to tropicalizations}, preprint, arXiv:1108.6126v2,
  2012.

\bibitem[Har77]{Hartshorne77}
R.~Hartshorne, \emph{Algebraic geometry}, Springer-Verlag, New York, 1977,
  Graduate Texts in Mathematics, No. 52.

\bibitem[Kaj08]{Kajiwara08}
T.~Kajiwara, \emph{Tropical toric geometry}, Toric topology, Contemp. Math.,
  vol. 460, Amer. Math. Soc., Providence, RI, 2008, pp.~197--207.

\bibitem[Kat09]{Katz09}
E.~Katz, \emph{A tropical toolkit}, Expo. Math. \textbf{27} (2009), no.~1,
  1--36.

\bibitem[Kat12]{Katz12}
\bysame, \emph{Tropical intersection theory from toric varieties}, Collect.
  Math. \textbf{63} (2012), no.~1, 29--44.

\bibitem[KK12]{KavehKhovanskii12}
K.~Kaveh and A.~Khovanskii, \emph{Newton-{O}kounkov bodies, semigroups of
  integral points, graded algebras and intersection theory}, Ann. of Math. (2)
  \textbf{176} (2012), no.~2, 925--978.

\bibitem[KP11]{KatzPayne11}
E.~Katz and S.~Payne, \emph{Realization spaces for tropical fans},
  Combinatorial aspects of commutative algebra and algebraic geometry, Abel
  Symp., vol.~6, Springer, Berlin, 2011, pp.~73--88.

\bibitem[Mat89]{Matsumura89}
H.~Matsumura, \emph{Commutative ring theory}, second ed., Cambridge Studies in
  Advanced Mathematics, vol.~8, Cambridge University Press, Cambridge, 1989.

\bibitem[Mik05]{Mikhalkin05}
G.~Mikhalkin, \emph{Enumerative tropical algebraic geometry in {$\Bbb R\sp
  2$}}, J. Amer. Math. Soc. \textbf{18} (2005), no.~2, 313--377.

\bibitem[Mik06]{Mikhalkin06}
\bysame, \emph{Tropical geometry and its applications}, International Congress
  of Mathematicians. Vol. II, Eur. Math. Soc., Z\"urich, 2006, pp.~827--852.

\bibitem[Nag66]{Nagata66}
M.~Nagata, \emph{Finitely generated rings over a valuation ring}, J. Math.
  Kyoto Univ. \textbf{5} (1966), 163--169.

\bibitem[Pay09a]{analytification}
S.~Payne, \emph{Analytification is the limit of all tropicalizations}, Math.
  Res. Lett. \textbf{16} (2009), no.~3, 543--556.

\bibitem[Pay09b]{tropicalfibers}
\bysame, \emph{Fibers of tropicalization}, Math. Z. \textbf{262} (2009), no.~2,
  301--311.

\bibitem[Pay12]{TropicalFibers-Correction}
\bysame, \emph{Erratum to: {F}ibers of tropicalization}, Math. Z. \textbf{272}
  (2012), no.~3-4, 1403--1406.

\bibitem[Poo93]{Poonen93}
B.~Poonen, \emph{Maximally complete fields}, Enseign. Math. (2) \textbf{39}
  (1993), no.~1-2, 87--106.

\bibitem[Rab12a]{Rabinoff12b}
J.~Rabinoff, \emph{Higher-level canonical subgroups for {$p$}-divisible
  groups}, J. Inst. Math. Jussieu \textbf{11} (2012), no.~2, 363--419.

\bibitem[Rab12b]{Rabinoff12}
\bysame, \emph{Tropical analytic geometry, {N}ewton polygons, and tropical
  intersections}, Adv. Math. \textbf{229} (2012), no.~6, 3192--3255.

\bibitem[RG71]{RaynaudGruson71}
M.~Raynaud and L.~Gruson, \emph{Crit\`eres de platitude et de projectivit\'e.
  {T}echniques de ``platification'' d'un module}, Invent. Math. \textbf{13}
  (1971), 1--89.

\bibitem[RGST05]{RST}
J.~Richter-Gebert, B.~Sturmfels, and T.~Theobald, \emph{First steps in tropical
  geometry}, Idempotent mathematics and mathematical physics, Contemp. Math.,
  vol. 377, Amer. Math. Soc., Providence, RI, 2005, pp.~289--317.

\bibitem[Ser65]{Serre65}
J.-P. Serre, \emph{Alg\`ebre locale. {M}ultiplicit\'es}, Lecture Notes in
  Mathematics, vol.~11, Springer-Verlag, Berlin, 1965.

\bibitem[Smi96]{Smirnov97}
A.~Smirnov, \emph{Torus schemes over a discrete valuation ring}, Algebra i
  Analiz \textbf{8} (1996), no.~4, 161--172.

\bibitem[Spe05]{SpeyerThesis}
D.~Speyer, \emph{Tropical geometry}, Ph.D. thesis, University of California,
  Berkeley, 2005.

\bibitem[SS04]{SpeyerSturmfels04}
D.~Speyer and B.~Sturmfels, \emph{The tropical {G}rassmannian}, Adv. Geom.
  \textbf{4} (2004), no.~3, 389--411.

\bibitem[ST08]{SturmfelsTevelev08}
B.~Sturmfels and J.~Tevelev, \emph{Elimination theory for tropical varieties},
  Math. Res. Lett. \textbf{15} (2008), no.~3, 543--562.

\bibitem[Tab08]{Tabera08}
L.~Tabera, \emph{Tropical resultants for curves and stable intersection}, Rev.
  Mat. Iberoam. \textbf{24} (2008), no.~3, 941--961.

\bibitem[Tev07]{Tevelev07}
J.~Tevelev, \emph{Compactifications of subvarieties of tori}, Amer. J. Math.
  \textbf{129} (2007), no.~4, 1087--1104.

\end{thebibliography}

\end{document}